\documentclass[11pt]{article}
\usepackage[T1]{fontenc}
\usepackage[latin9]{inputenc}
\usepackage[pdftex]{graphicx,color}
\usepackage{listings, textcomp}
\usepackage{amsmath,amssymb,amsfonts,amsthm,amssymb,latexsym}

\evensidemargin=- 0.3 cm
\newtheorem{theorem}{Theorem}[section]

\newtheorem{remark}{Remark}[section]

\def \ve{\varepsilon}

\begin{document}



%
%

\title{Hopf Bifurcation in a Gene Regulatory Network Model: \\ Molecular Movement Causes Oscillations }

\author{{Mark Chaplain$\,^{\mathrm{a}}$,  Mariya Ptashnyk$\,^{\mathrm{a}}$ and Marc Sturrock$\,^{\mathrm{b}}$}\medskip\\
\small $^{\mathrm{a}}$ 
Division of Mathematics, 
University of Dundee, 
Dundee DD1 4HN, 
Scotland, UK
\\
\small $^{\mathrm{b}}$Mathematical Biosciences Institute,
Ohio State University, 
377 Jennings Hall,
\\
\small 1735 Neil Avenue, 
Columbus, OH,
USA}

\maketitle

\begin{abstract}
Gene regulatory networks, i.e. DNA segments in a cell which interact with each other indirectly through their RNA and protein products, lie at the heart of many important intracellular signal transduction processes. In this paper we analyse a mathematical model of a canonical gene regulatory network consisting of a single negative feedback loop between a protein and its mRNA (e.g. the Hes1 transcription factor system). The model consists of two partial differential equations describing the spatio-temporal interactions between the protein and its mRNA in a 1-dimensional domain. Such intracellular negative feedback systems are known to exhibit oscillatory behaviour and this is the case for our model, shown initially via computational simulations. In order to investigate this behaviour more deeply, we next solve our system using Green's functions and then undertake a linearized stability analysis of the steady states of the model. Our results show that the diffusion coefficient of the protein/mRNA acts as a bifurcation parameter and gives rise to a Hopf bifurcation. This shows that the spatial movement of the mRNA and protein molecules alone is sufficient to cause the oscillations. This has implications for transcription factors such as p53, NF-$\kappa$B and heat shock proteins which are involved in regulating important cellular processes such as inflammation, meiosis, apoptosis and the heat shock response, and are linked to diseases such as arthritis and cancer. 
\end{abstract}



\section{Introduction}

A gene regulatory network (GRN) can be defined as a collection of DNA segments in a cell which interact with each other indirectly through their RNA and protein products. GRNs lie at the heart of intracellular signal transduction and indirectly control many important cellular functions. A key component of GRNs is a class of proteins called transcription factors. In response to various biological signals, transcription factors change the transcription rate of genes, allowing cells to produce the proteins they need at the appropriate times and in the appropriate quantities. It is now well stablished that GRNs contain a small set of recurring regulation patterns, commonly referred to as network motifs \cite{Milo}, which can be thought of as recurring circuits of interactions from which complex GRNs are built. A GRN is said to contain a negative feedback loop if a gene product inhibits its own production either directly or indirectly. Negative feedback loops are commonly found in diverse biological processes including inflammation, meiosis, apoptosis and the heat shock response \cite{Alberts,Tyson,Lahav}, and are known to exhibit oscillations in mRNA and protein levels \cite{Zatorsky2010,Kobayashi2011,Nelson}.

Mathematical modelling of GRNs goes back to the work of Goodwin \cite{Goodwin}, where an initial system of two ordinary differential equations (ODEs) was used to model a self-repressing gene. In the final part of the paper a system of 3 ODEs was shown to produce limit cycle behaviour. This work was continued by Griffith \cite{Griffithnf} who demonstrated that the introduction of the third species was necessary for the oscillatory dynamics. An analysis of theoretical chemical systems whereby two chemicals produced at distinct spatial locations (heterogeneous catalysis) diffused and reacted together was carried out by Glass and co-workers \cite{glasskauffman,Shymko}. Their results showed that the number and stability of the steady states of the system changed depending on the distance between the two catalytic sites. The authors concluded that {\it ``These examples indicate that geometrical considerations must be explicitly considered when analyzing the dynamics of highly structured (e.g., biological) systems.''} \cite{Shymko}. Mahaffy and co-workers \cite{Mahaffy1985,Mahaffy1988,Mahaffy1984} developed this work further by considering an explicitly spatial model and also time delays accounting for the processes of transcription (production of mRNA) and translation (production of proteins). Tiana et al. \cite{Tiana} proposed that introducing delays to ODE models of negative feedback loops could produce sustained oscillatory dynamics and Jensen et al. \cite{Jensen} found that the invocation of an unknown third species (as Griffith had done \cite{Griffithnf}) could be avoided by the introduction of delay terms to a model of the Hes1 GRN (the justification being to account for the processes of transcription and translation). The Hes1 system is a simple example of a GRN which possesses a single negative feedback loop and benefits from having been the subject of numerous biological experiments \cite{Hirata,Kageyama,Kobayashi2010,Kobayashi2011,Kobayashi2009}. A delay differential equation (DDE) model of the Hes1 GRN has also been studied by Monk and co-workers \cite{Monk2008,Monk2003}. More recently a 2-dimensional spatio-temporal model of the Hes1 GRN considering diffusion of the protein and mRNA was developed by Sturrock et al. \cite{Sturrock2011} and then later extended to account for directed transport via microtublues \cite{Sturrock2012}. 

A key feature of all mathematical models of the Hes1 GRN (and other negative feedback systems) is the existence of oscillatory solutions characterised by a Hopf bifurcation. In the Hopf, or Poincar\'e-Andronov-Hopf bifurcation (first described by Hopf \cite{Hopf}), a steady state changes stability as two complex conjugate eigenvalues of the linearization cross the imaginary axis and a family of periodic orbits bifurcates from the steady state. Many studies are devoted to the existence and stability of Hopf bifurcations in ordinary and partial differential equations \cite{Crandall,Hassard,Kielhoefer,Kielhoefer_paper,Kuznetsov,Marsden}. The question of the existence of global Hopf bifurcation for nonlinear parabolic equations has also been considered \cite{Fiedler2,Fiedler,Ize}. There are many results concerning the stability of constant (i.e. spatially homogeneous) steady states and the existence of periodic solutions bifurcating from such constant steady states. There are some results on the stability of spike-solutions and the existence of Hopf bifurcations in the shadow Gierer-Meinhardt model \cite{Dancer2,Ni,Ward,Ward2}, as well as on the stability of spiky solutions in a reaction-diffusion system with four morphogens \cite{Wei2} and of cluster solutions for large reaction-diffusion systems \cite{Wei}. In the analysis of  the stability and Hopf bifurcations in systems with spike-solutions as stationary solutions, the properties of the corresponding nonlocal eigenvalue problem were used. Perturbation theory has been applied to analyse the stability of non-constant steady-states for a system of nonlinear reaction-diffusion equations coupled with ordinary differential equations \cite{Golovaty}. In considering the relation between the spectrum of a linearised operator for singularly perturbed predator-prey-type equations with diffusion and the limit operator as the perturbation parameter tends to zero, Dancer \cite{Dancer} analysed the stability of strictly positive stationary solutions and the existence of Hopf bifurcations.

In this paper we analyse a mathematical model of the Hes1 transcription factor - a canonical GRN consisting of a single negative feedback loop between the Hes1 protein and its mRNA. The format of this paper is as follows. In the next section we present our mathematical model derived from that first formulated by Sturrock et al. \cite{Sturrock2011}. First we demonstrate the existence of oscillatory solutions numerically, indicating the existence of Hopf bifurcations. Next, applying linearised stability analysis, we study the stability of a (spatially inhomogeneous) steady state of the model and prove the existence of a Hopf bifurcation. The main difficulty of the analysis is that the steady state of the model is not constant. In a similar manner to Dancer \cite{Dancer} we show the existence of a Hopf bifurcation by considering a limit problem associated with the original model. The method of collective compactness \cite{Anselone,Dancer} is applied to relate the spectrum of the limit operator to the spectrum of the original operator. To show the stability of periodic solutions and to determine the type of Hopf bifurcation, we use a weakly nonlinear analysis, see, for example, Matkowski \cite{Matkowski}, and normal form theory, see, for example, Hassard, 
Haragus \cite{Haragus,Hassard}. The techniques of weakly nonlinear analysis \cite{Eftimie,Maini,Lewis,Rodrigues} and normal form theory \cite{Haragus,Hassard}, are widely used to study the nonlinear behaviour of solutions near bifurcation points. 

\section{The Mathematical Model of the Hes1 Gene Regulatory Network} 

The basic model of a self-repressing gene \cite{Monk2003} describing the temporal dynamics of hes1 mRNA concentration, $m(t)$, and Hes1 protein concentration, $p(t)$, takes the general form:
\begin{eqnarray}\label{hes1m}
\frac{\partial m }{\partial t} & = & \alpha_m f(p) - \mu_m m , \\
\frac{\partial p}{\partial t}  & = & \alpha_p m - \mu_p p ,
\label{hes1p}
\end{eqnarray} 
for positive constants $\alpha_m , \alpha_p , \mu_m , \mu_p$ and some function $f(p)$ modelling the suppression of mRNA production by the protein. It can be shown using Bendixson's Negative Criterion (cf. Verhulst \cite{Verhulst}, Theorem 4.1) that, irrespective of the function $f(p)$ (e.g. a Hill function), there are no periodic solutions of the above system. In order to account for the experimentally observed oscillations in both mRNA and protein concentration levels \cite{Hirata}, a discrete delay has often been introduced into such models being justified as taking into account the time taken to produce mRNA (transcription) and produce protein (translation) \cite{Monk2003}. Applying a discrete delay $\tau$ to (\ref{hes1m}), (\ref{hes1p}), a delay differential equation model is obtained of the form:
\begin{eqnarray}
\frac{\partial m }{\partial t} & = & \alpha_m f(p - \tau) - \mu_m m , \\
\frac{\partial p}{\partial t}  & = & \alpha_p m - \mu_p p .
\end{eqnarray}
Such a system is observed to exhibit oscillations for a suitable value of the delay parameter $\tau$ representing the sum of the transcriptional and translational time delays. This delay differential equation approach has also been used to model other feedback systems involving transcription factors such as p53 \cite{Bar-Or,Zatorsky,Tiana} and NF-$\kappa$B \cite{Nelson}. Other papers have used a distributed delay to model this effect\cite{Gordon}, which in fact is equivalent to the original three ODE model of a self-repressing gene proposed by Goodwin \cite{Goodwin} and Griffith \cite{Griffithnf}. 

Here we study an explicitly spatial model of the Hes1 GRN originally formulated by Sturrock et al. \cite{Sturrock2011,Sturrock2012} and investigate the role that spatial movement of the molecules may play in causing the oscillations in concentration levels. The model consists of a system of coupled nonlinear partial differential equations describing the temporal and spatial dynamics of the concentration of hes1 mRNA, $m(x,t)$, and Hes1 protein, $p(x,t)$, and accounts for the processes of transcription (mRNA production) and translation (protein production). Transcription is assumed to occur in a small region of the domain representing the gene site. Both mRNA and protein also diffuse and undergo linear decay. The  non-dimensionalised model is given as:
\begin{align}\label{dynamic_problem}
\begin{aligned}
&\frac{\partial m }{\partial t}  = D \frac{\partial^2 m }{\partial x^2} + \alpha_m \, f(p )\delta^\ve_{x_M}(x) - \mu \,  m  \quad & \text{ in } (0,T)\times (0,1), \\
&\frac{\partial p}{\partial t}  = D \frac{\partial^2 p }{\partial x^2}+ \alpha_p \, g(x) \, m - \mu\,  p  \quad & \text{ in } (0,T)\times (0,1), \\
&\frac{\partial m(t,0)}{\partial x}  =\frac{\partial m(t,1)}{\partial x}  =0, \quad \frac{\partial p(t,0)}{\partial x} = \frac{\partial p(t,1)}{\partial x} =0 \quad & \text{ in } (0,T),\\
& m(0,x)= m_0(x) , \quad p(0,x) = p_0(x) \quad & \text{ in } (0,1),
\end{aligned}
\end{align} 
where $D$, $\alpha_m$, $\alpha_p$ and $\mu$ are positive constants (the diffusion coefficient, transcription rate, translation rate and decay rate respectively). Here $l$ denotes the position of the nuclear membrane and therefore the domain is partitioned into two distinct regions, $(0,l)$ the cell nucleus and $(l,1)$ the cell cytoplasm, for some $l\in (0,1)$. The point $x_M \in (0, l)$ is the position of the centre of the gene site and by $\delta^\ve_{x_M}$ we denote the Dirac approximation of the $\delta$-distribution located at $x_M$, with $\ve>0$ a small parameter and $\delta^\ve_{x_M}$ has compact support.    

The nonlinear reaction term $f: \mathbb R \to \mathbb R$ is a Hill function  $f(p) = {1}/{(1+p^h)}$, with $h \geq 1$, modelling the suppression of mRNA production by the protein (negative feedback). The function $g$ is a step function  given by 
$$
g(x) = \left\{ \begin {array}{ll} 0, & \mbox{if}\,\, x < l \; , \\
 \\
1, & \mbox{if}\,\, x  \geq l\; ,\\
\end{array} \right.
$$
since the process of translation only occurs in the cytoplasm. 
A schematic diagram of the domain is given in Figure \ref{fig1a}.

\begin{figure}[h!]
\centering
\includegraphics[scale=1.0]{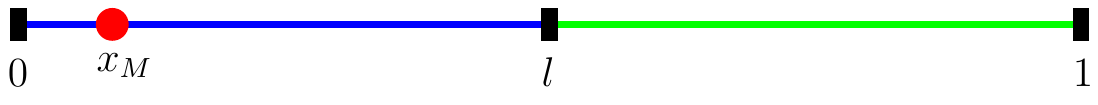}               
\caption{Schematic diagram of the 1-dimensional spatial domain for the system (\ref{dynamic_problem}), showing the spatial location of the gene site $x_M$ and the nuclear membrane $l$. The cell nucleus is shown in blue, while the cell cytoplasm is shown in green.}
\label{fig1a}
\end{figure}

First we demonstrate existence and uniqueness of solutions to  \eqref{dynamic_problem}. 
 
\begin{theorem}
For $\ve>0$ and nonnegative  initial data $m_0, p_0  \in H^2(0,1)$,   there exists  a unique nonnegative global solution  $m,p\in C([0, \infty); H^2(0,1))$, $\partial_t m, \partial_t p \in L^2((0,T)\times(0,1))$,  and $m,p \in C^{(\gamma+1)/2, \gamma +1}([0, T]\times [0,1])$, for some $\gamma>0$ and any $T>0$,  of the problem \eqref{dynamic_problem} satisfying 
\begin{eqnarray}\label{a_priori_estim}
\begin{aligned}
\|m\|_{L^\infty(0,T; H^1(0,1))}+   \|p\|_{L^\infty (0,T; H^1(0,1))}  \leq C, \\ 
 \|\partial_t m\|_{L^2((0,T)\times (0,1))} + \|\partial_t p\|_{L^2(0,T; H^1(0,1))}   +
  \|\partial_x^2 p\|_{L^2((0,T)\times(0,1))} \leq C, 
\end{aligned}
\end{eqnarray}
for any $T\in (0, \infty)$ with the constant $C$ independent of $\ve$. 
\end{theorem} 

\begin{proof}
Since $f(p)$ is Lipschitz continuous for  $p\geq -\theta $, with some $0<\theta <1$, we have  that for nonnegative initial data $m_0$, $p_0$ the existence and uniqueness of a solution of the problem  \eqref{dynamic_problem}  in $(0, T_0)\times (0,1)$, for some $T_0 >0$,  follows directly from the existence and the regularity theory for systems of parabolic equations, see e.g. Henry \cite{Henry}, Lieberman \cite{Lieberman}. 
Using the definition  of the Dirac sequence, for 
$
F_m(m,p) = \alpha_m \, f(p )\delta^\ve_{x_M}(x) - \mu \,  m 
$
and 
$F_p(m,p) =\alpha_p \, g(x) \, m - \mu\,  p,$
we have 
\begin{eqnarray*}
F_m |_{m=0} \geq 0 \quad \text{ for } p \geq 0, \quad &&
F_p |_{p=0} \geq 0 \quad  \hspace{1.5 cm }  \text{ for } m \geq 0, \\
F_m |_{m=\alpha_m/(\mu \ve)} \leq 0 \quad \text{ for } p \geq 0, \quad &&
F_p |_{p=\alpha_m\alpha_p/(\mu^2 \ve)} \leq 0 \quad \text{ for } m\leq \alpha_m/(\mu \ve) \; .
\end{eqnarray*}
Thus applying the theorem of invariant regions, e.g. Theorem 14.7 in Smoller \cite{Smoller}, with $G_1(m,p) = - m$, $G_2(m,p) = - p$,  
$G_3(m,p) = m-\alpha_m/(\mu \ve)$, and $G_4(m,p) = p-\alpha_m\alpha_p/(\mu^2 \ve)$, we conclude that $0 \leq m(t,x)\leq \alpha_m/(\mu \ve)$ and $0 \leq p(t,x) \leq \alpha_m\alpha_p/(\mu^2 \ve)$ for all $(t,x) \in (0, T_0)\times(0,1)$, whereas the bounds for $m$ and $p$ are uniform in $T_0$.  This ensures global existence and uniqueness of a bounded solution of  \eqref{dynamic_problem}  for fixed $\ve$. 

Using the property of the Dirac sequence, i.e. $\|\delta^\ve_{x_M} \|_{L^1(0,1)} =1$,  continuous embedding of  $H^1(0,1)$ in $C([0,1])$,  and considering  $m$ and $p$ as test functions for \eqref{dynamic_problem} we obtain
\begin{equation*}\label{apriori_2}
\begin{aligned}
\partial_t \|m(t)\|^2_{L^2(0,1)} + \|\partial_x m(t)\|^2_{L^2(0,1)} + \| m(t) \|^2_{L^2(0,1)} &\leq C \|f(p) \|^2_{L^\infty((0,T)\times (0,1))}\; , \\
\partial_t \|p(t)\|^2_{L^2(0,1)} + \|\partial_x p(t)\|^2_{L^2(0,1)} + \| p(t) \|^2_{L^2(0,1)} &\leq C \| m (t)\|^2_{L^2(0,1)} \; .
\end{aligned}
\end{equation*} 
Integrating  over time and using the uniform boundedness of $f(p)$ for nonnegative $p$ ensure the estimates in  $L^\infty(0,T; L^2(0,1))$ and  $L^2(0,T; H^1(0,1))$. 

Testing  the first equation in \eqref{dynamic_problem} with  $\partial_t m$   and  the second equation with
$\partial_t p$ and $\partial^2_x p$, as well as differentiating the second equation with respect to $t$ and testing with $\partial_t p$,  and integrating over $(0, \tau)$  for  $\tau \in (0,T)$ and any  $T>0$ imply
\begin{eqnarray*}
&& \|\partial_t m\|^2_{L^2((0,\tau)\times(0,1))} +  \|\partial_x m(\tau)\|^2_{L^2(0,1)} + \| m(\tau) \|^2_{L^2(0,1)}
 \leq \delta \|m(\tau) \|^2_{L^\infty(0,1)}\\
&& \hspace{2 cm }  + C\left[\|m(0)\|^2_{H^1(0,1)}+ \|m \|^2_{L^2(0,\tau;L^\infty(0,1))} 
  + \|\partial_t p \|^2_{L^2(0,\tau;L^\infty(0,1))} + C_\delta\right] , \\
&& \|\partial_t p\|^2_{L^2((0,\tau)\times(0,1))} +  \|\partial_x p(\tau)\|^2_{L^2(0,1)} \leq C  \left[\| m \|^2_{L^2((0,\tau)\times(0,1))}+ \|p(0)\|^2_{H^1(0,1)}\right] \; , \\
&&   \|\partial_x p(\tau)\|^2_{L^2(0,1)} + \|\partial_x^2 p\|^2_{L^2((0,\tau)\times(0,1))}  \leq C
 \left[ \|m \|^2_{L^2((0,\tau)\times(0,1))}+ \|p(0)\|^2_{H^1(0,1)}\right] \; , \\
&&  \|\partial_t p(\tau)\|^2_{L^2(0,1)} + \|\partial_x \partial_t p\|^2_{L^2((0,\tau)\times(0,1))}  \leq  \delta \|\partial_t  m \|^2_{L^2((0,\tau)\times(0,1))} \\
&& \hspace{5.7 cm } +  C_\delta\left[\|\partial_t p\|^2_{L^2((0,\tau)\times(0,1))} + \|\partial_t p(0)\|^2_{L^2(0,1)}\right]\; .
\end{eqnarray*} 
This together with  the continuous embedding of $H^1(0,1)$ in $C([0,1])$, the estimate 
$\|\partial_t p(0)\|_{L^2(0,1)}\leq C \|p(0)\|_{H^2(0,1)}$, regularity of initial data and estimates in   $L^\infty(0,T; L^2(0,1))$ and  $L^2(0,T; H^1(0,1))$ shown above ensures  estimates  \eqref{a_priori_estim}.
 \end{proof}
 
\begin{remark} The {\it a priori} estimates   \eqref{a_priori_estim} imply the uniform in $\ve$ boundedness of solutions  of  \eqref{dynamic_problem} for every $T >0$. 
   \end{remark}

For the qualitative analysis of  \eqref{a_priori_estim} we consider the following parameter values in the model  equations:
the basal transcription rate of hes1 mRNA is given by $\alpha_m=1$, the translation rate of Hes1 protein is $\alpha_p =2$, the Hill coefficient in the function $f$ is taken to be $h=5$, and the degradation rate of hes1 mRNA/Hes1 protein $\mu =0.03$. It is assumed that the region of the cytoplasm where the protein is produced is given by $(1/2, 1)$, i.e. $l=1/2$, and the position of the centre of the gene site is at $x_M=0.1$. The diffusion coefficient is a variable parameter in the model and we consider a range of (non-dimensional) diffusion coefficients $D \in [d_1, d_2]$, where $d_1 =10^{-7}$ and $d_2=0.1$,  arising from a corresponding range of biologically relevant dimensional 
values \cite{Matsuda,Mendez,Seksek}. 

Numerical simulations of the model \eqref{dynamic_problem} (using the forward Euler scheme in time and a centred difference scheme in space, as well as 
the Dirac sequence in the form $\delta^\ve_{x_M}(x)=\frac 1{2\ve}(1+\cos(\pi(x-x_M)/\ve))$ for $|x-x_M|<\ve$ and $\delta^\ve_{x_M}(x) = 0$ for $|x-x_M|\geq \ve$) 
reveal that a stationary solution, stable for small values of the diffusion coefficient $D$, becomes unstable for $D\geq D_{1,\ve}^c$, with  $D_{1,\ve}^c\approx 3.117 \times 10^{-4}$, and  again stable for $D> D^c_{2,\ve}$, where $D^c_{2, \ve}\approx 7.885 \times  10^{-3}$. For diffusion coefficients between the two critical values, i.e. $D \in [D_{1, \ve}^c, D_{2,\ve}^c]$,  numerical simulations show the existence of stable periodic solutions of the model  \eqref{dynamic_problem}. These scenarios are shown in Figs.~\ref{fig1}-\ref{fig4}.

  \begin{figure}[h!]
  \centering
  \includegraphics[width=0.4\textwidth]{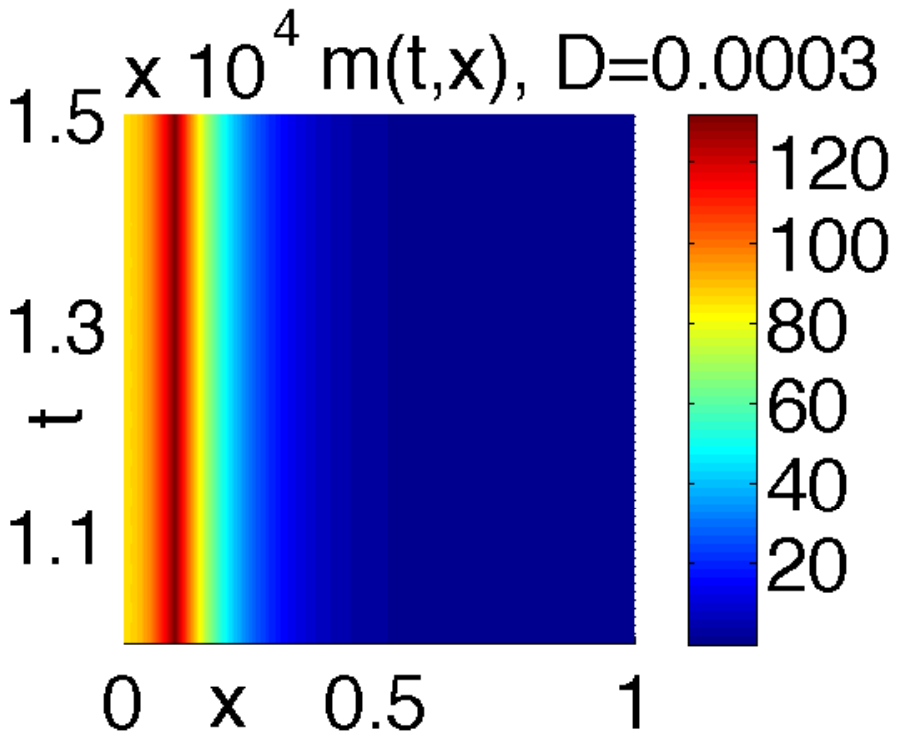}\; \;                
\includegraphics[width=0.4\textwidth]{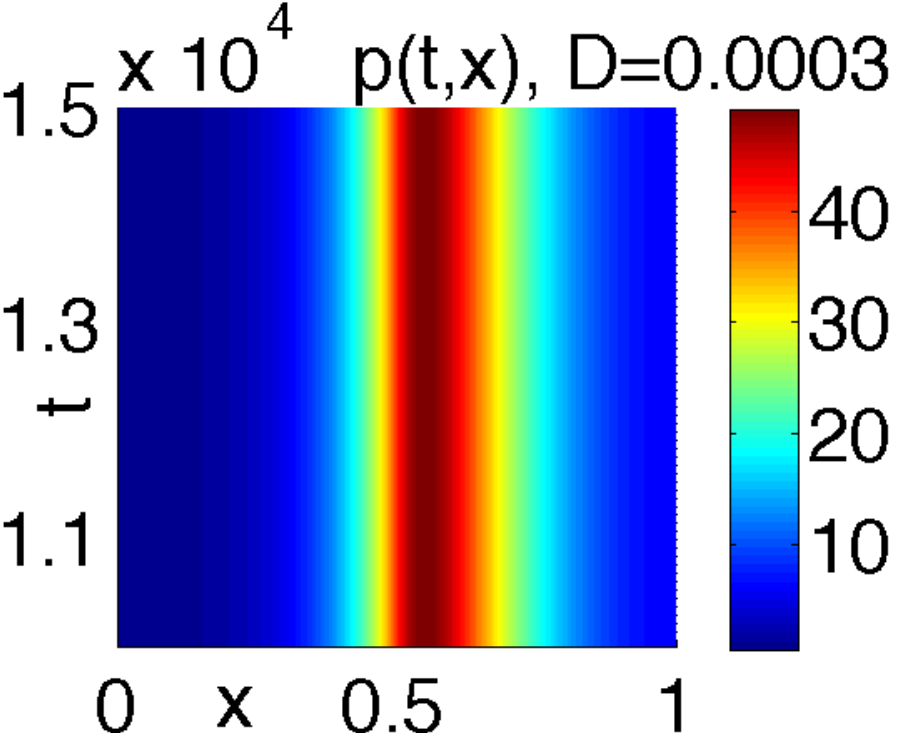}

 \includegraphics[width=0.4\textwidth]{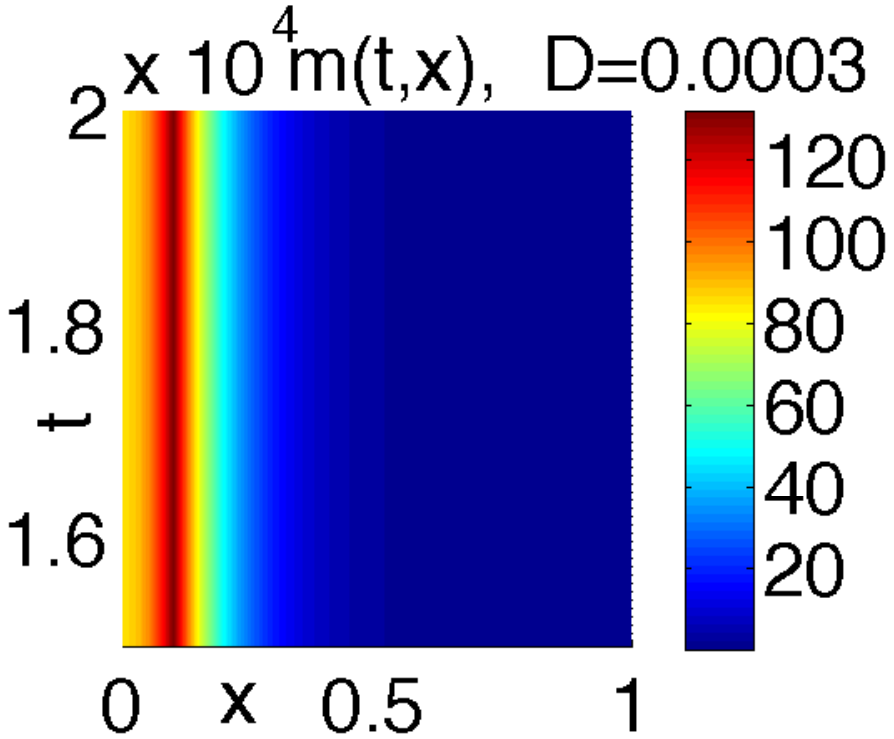}    \; \;            
\includegraphics[width=0.4\textwidth]{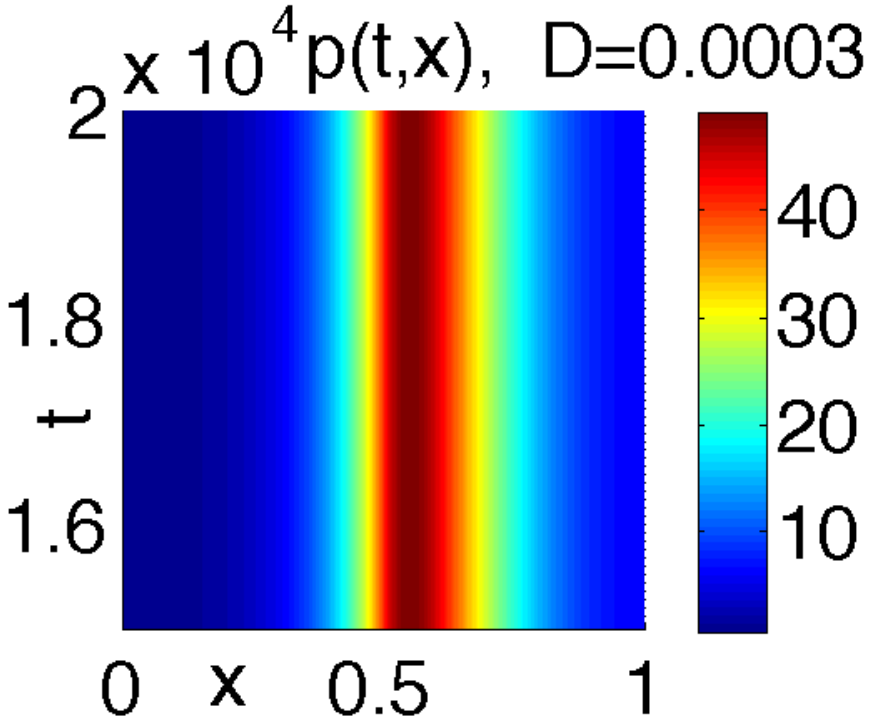}

\includegraphics[width=0.3\textwidth]{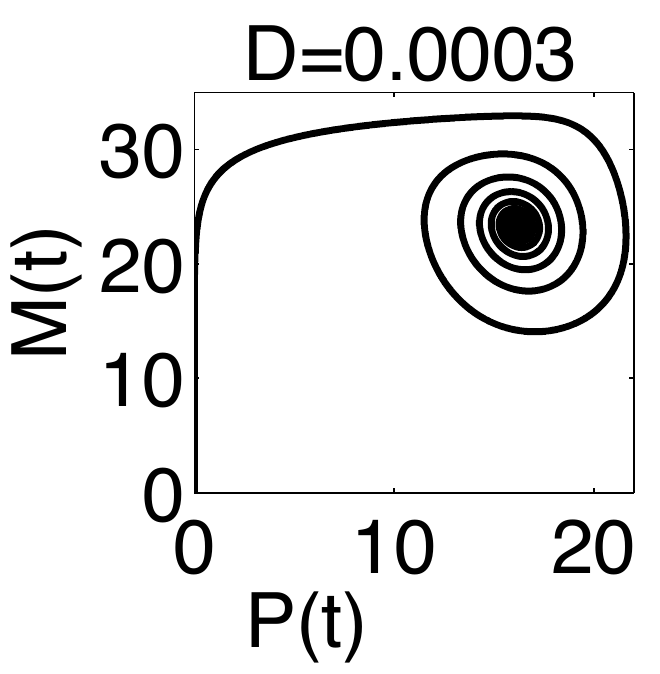}\; \; \; \; \; \; 
\includegraphics[width=0.3\textwidth]{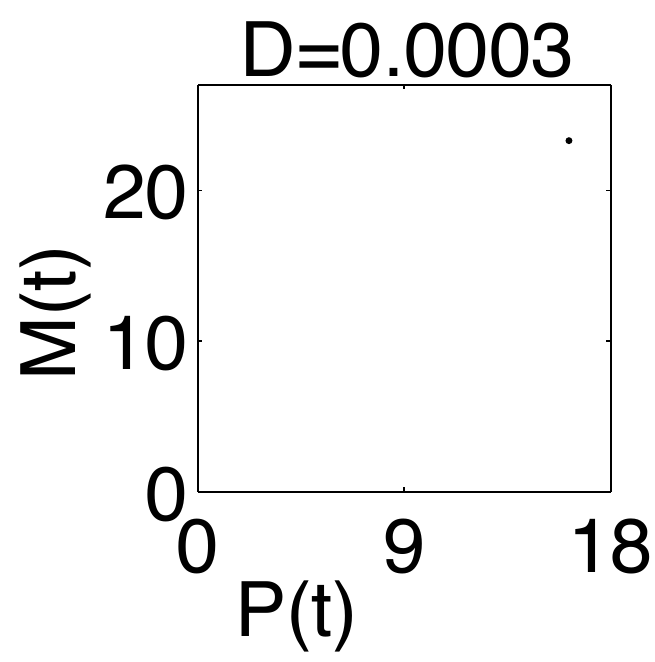}

  \caption{First two rows: plots showing the spatio-temporal evolution of mRNA level, $m(t,x)$, and protein level, $p(t,x)$, from numerical  simulations of system \eqref{dynamic_problem} with zero initial conditions, with $\ve=10^{-3}$, $D=0.0003$, and $t \in[10^4, 2\cdot 10^4]$. The plots show that the solutions tend to a steady-state. Bottom row: the corresponding phase-plots, where $M(t)=\int_0^1 m(t,x) dx$ and  $P(t)=\int_0^1 p(t,x) dx$. The figure on the left is for $t\in [0, 2 \times 10^4]$, and the figure on the right is for $t\in[10^4, 2 \times 10^4]$. These show the trajectory converging to a fixed point, equivalent to the steady-state.}
  \label{fig1}
\end{figure}

   \begin{figure}[h!]
  \centering
    \includegraphics[width=0.4\textwidth]{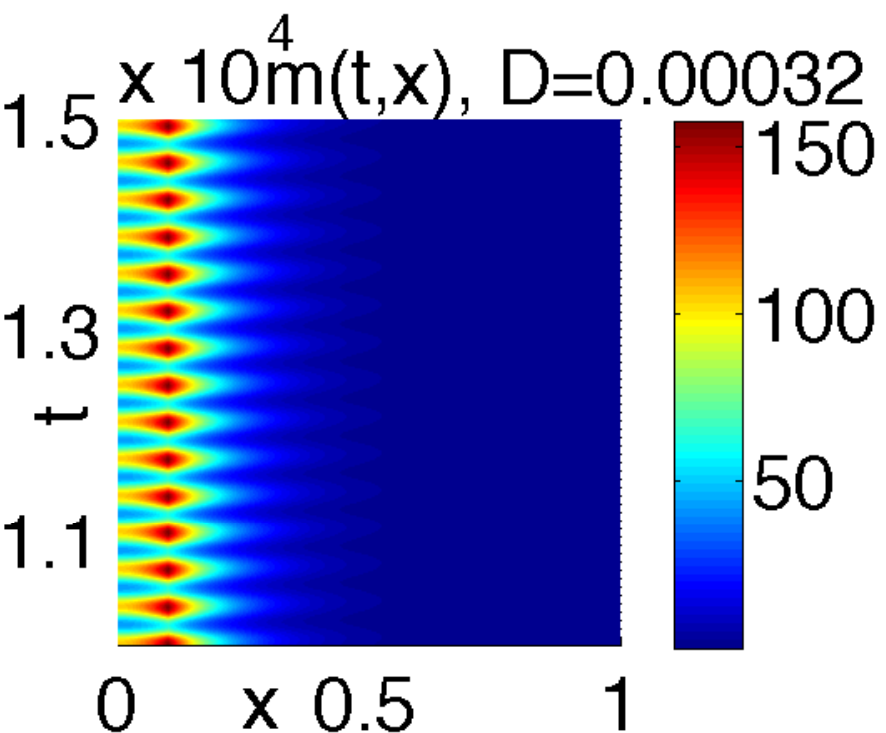}   \; \;             
\includegraphics[width=0.4\textwidth]{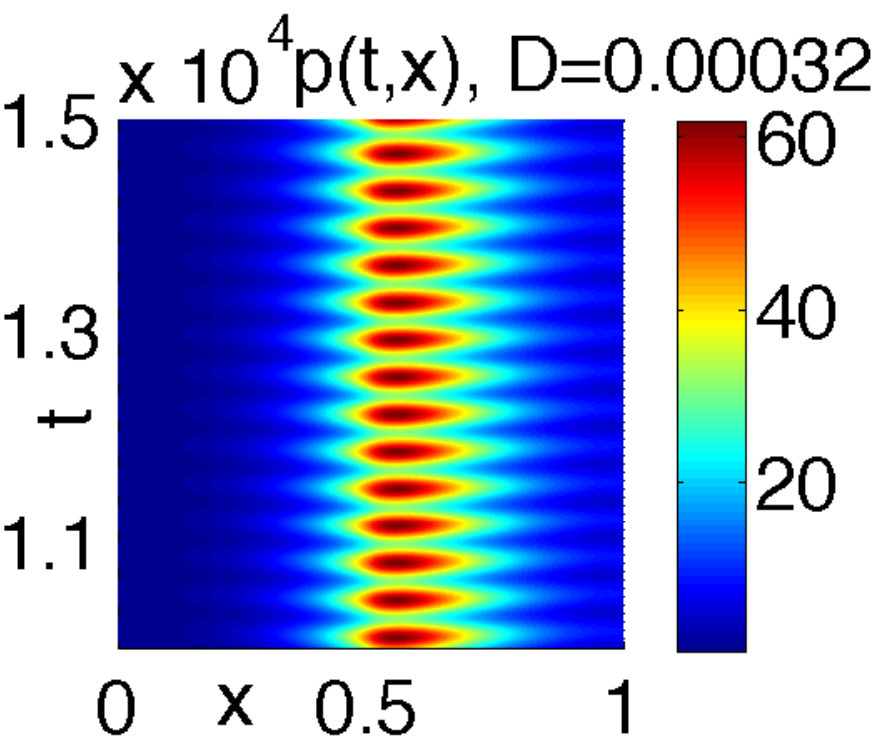}

  \includegraphics[width=0.4\textwidth]{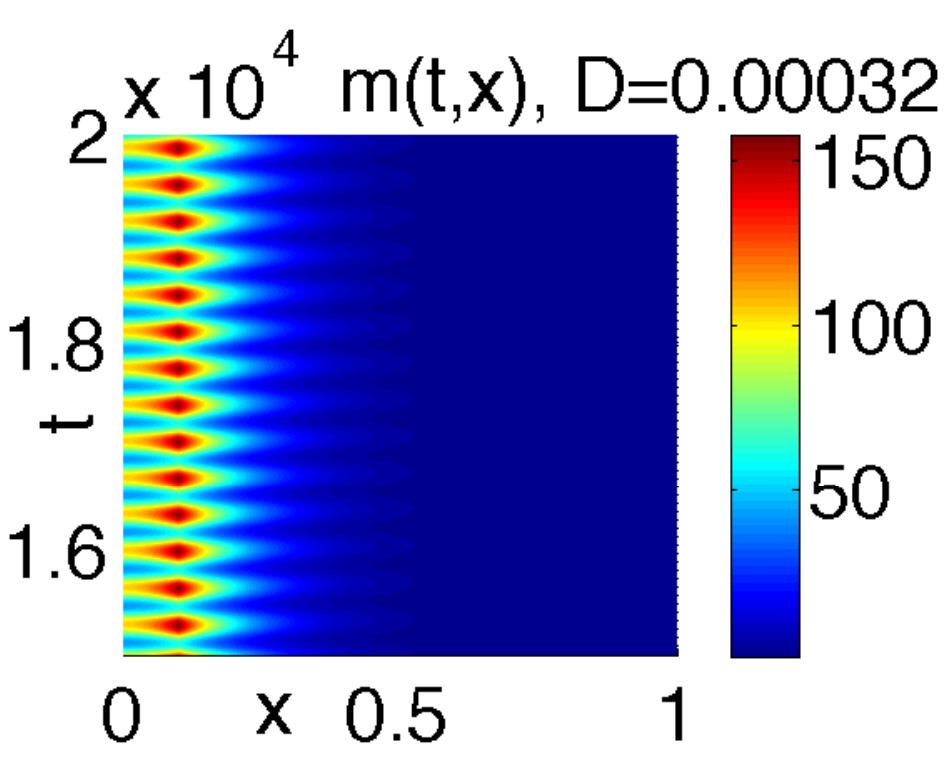}  \; \;              
\includegraphics[width=0.4\textwidth]{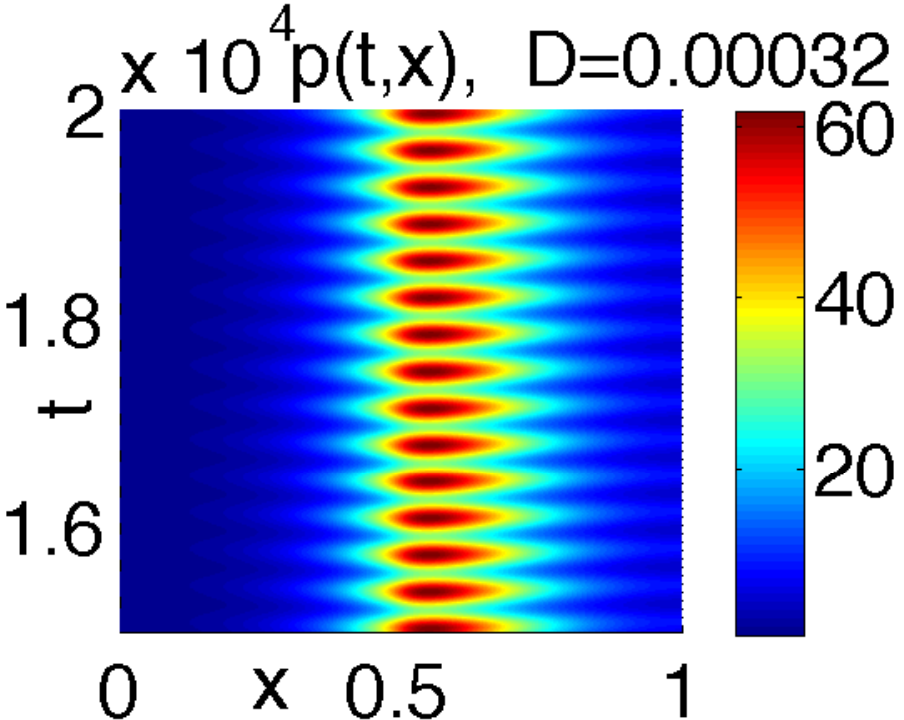}

\includegraphics[width=0.3\textwidth]{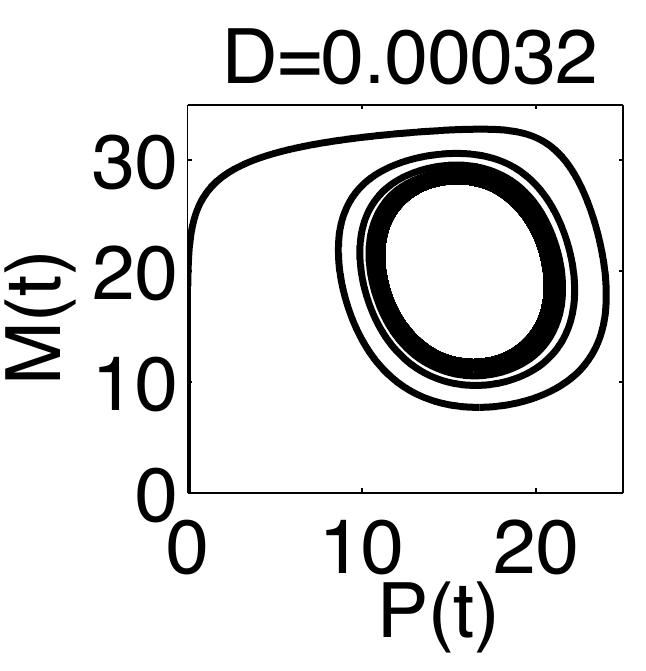} \; \; \; \; \; \; 
\includegraphics[width=0.3\textwidth]{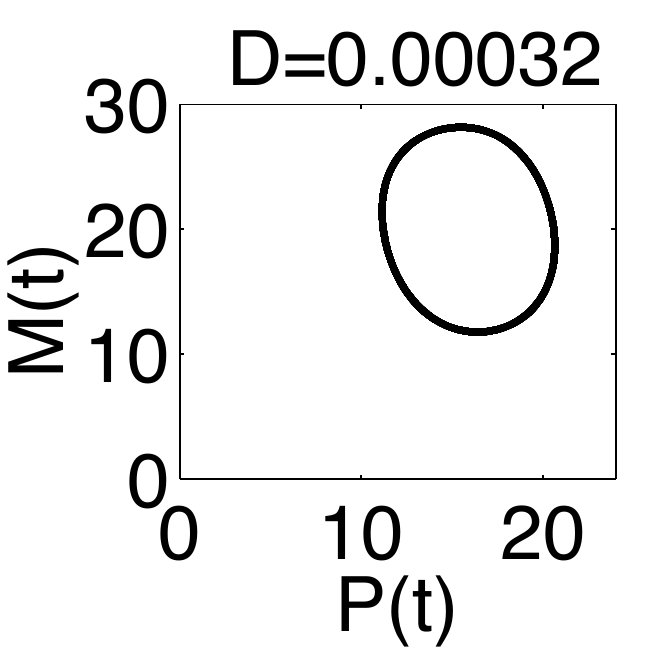}
  \caption{First two rows: plots showing the spatio-temporal evolution of mRNA level, $m(t,x)$, and protein level, $p(t,x)$, from numerical  simulations of system \eqref{dynamic_problem} with zero initial conditions, with $\ve=10^{-3}$, $D=0.00032$ and $t \in[10^4, 2 \times 10^4]$. The plots show oscillatory solutions. Bottom row: the corresponding phase-plots, where $M(t)=\int_0^1 m(t,x) dx$ and  $P(t)=\int_0^1 p(t,x) dx$. The figure on the left is for $t\in [0, 2 \times 10^4]$, and the figure on the right is for $t\in[10^4, 2 \times 10^4]$. These show the trajectory converging to a limit-cycle.}
  \label{fig2}
\end{figure}

 \begin{figure}[h!]
  \centering
  \includegraphics[width=0.405\textwidth]{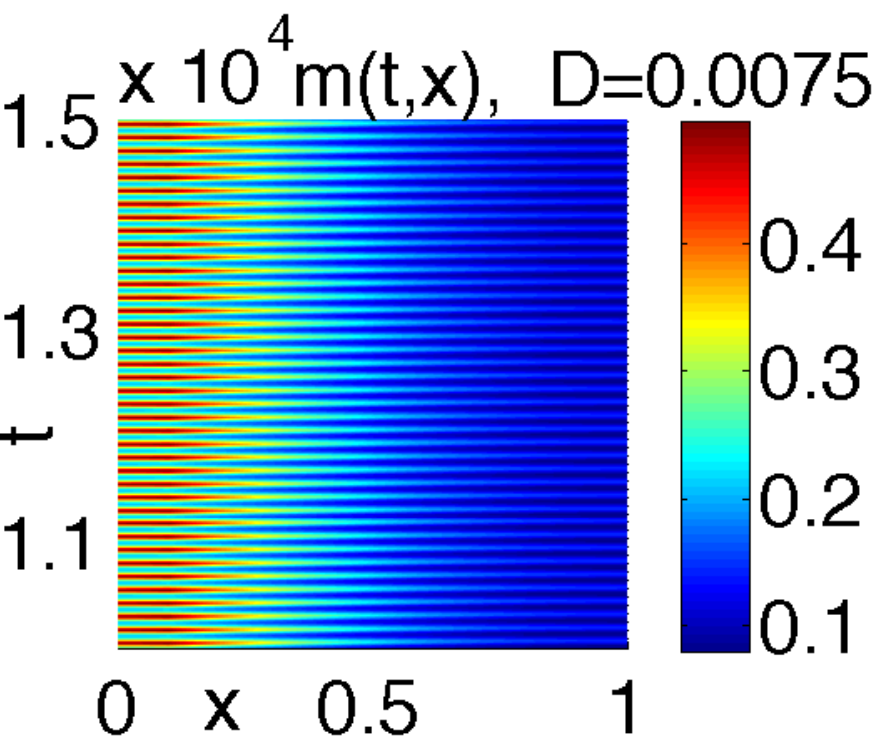}       \; \;         
\includegraphics[width=0.395\textwidth]{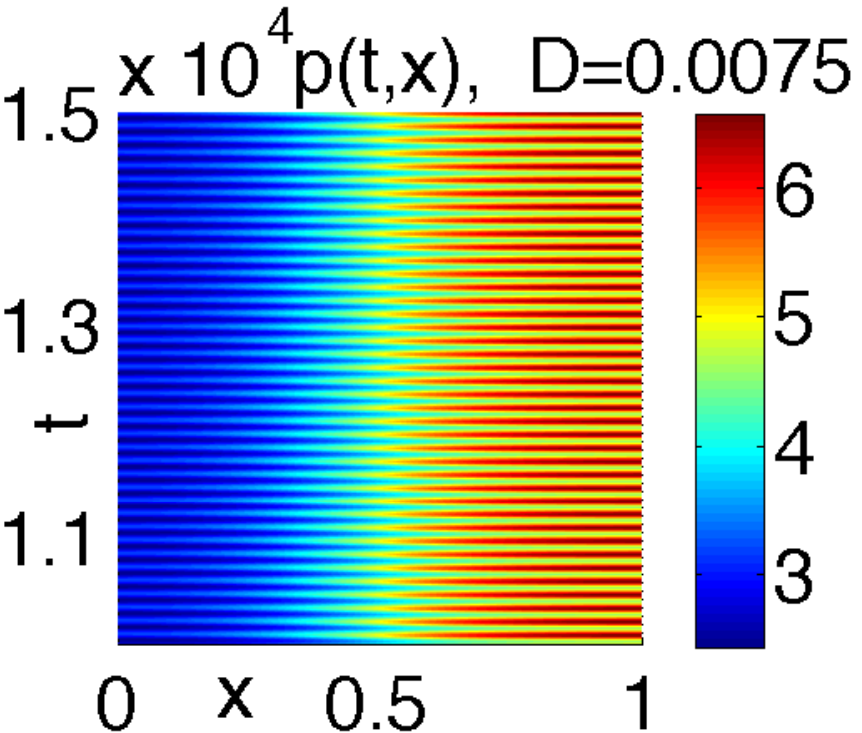}

 \includegraphics[width=0.425\textwidth]{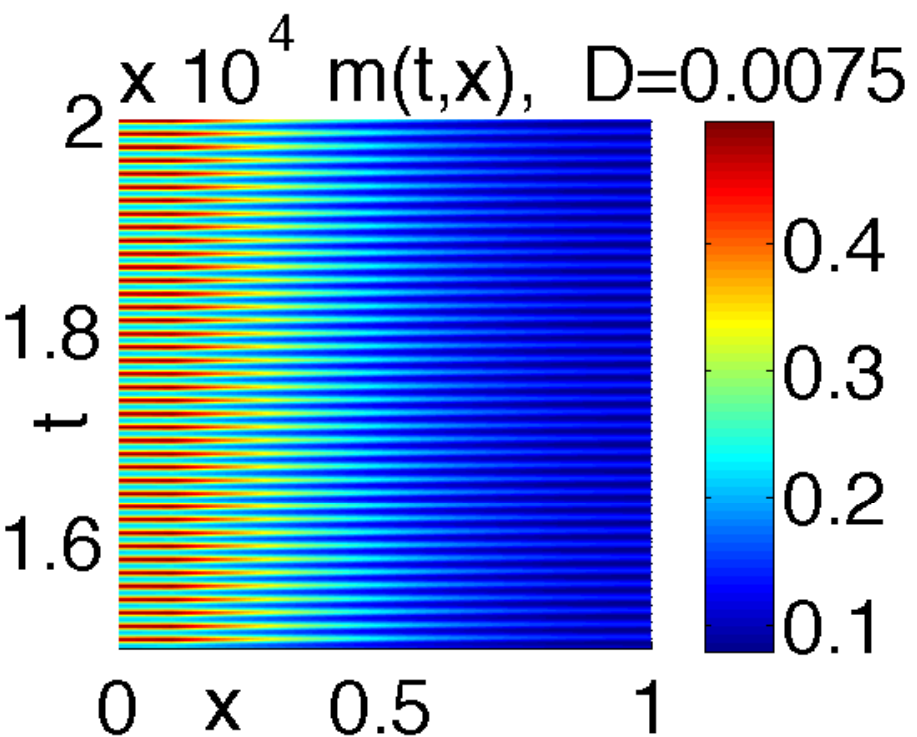}      \; \;          
\includegraphics[width=0.395\textwidth]{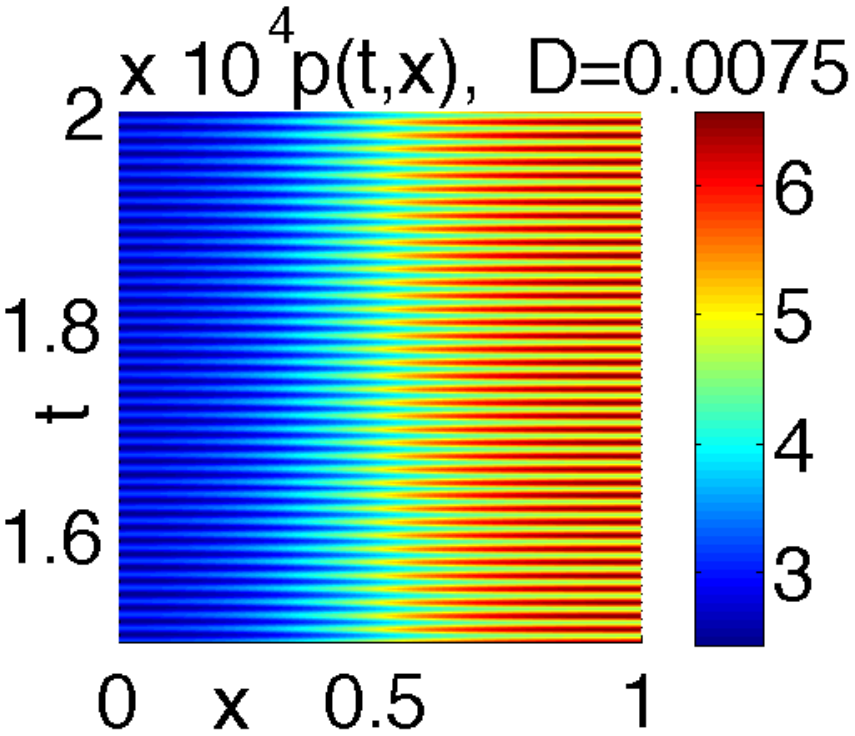}

\includegraphics[width=0.31\textwidth]{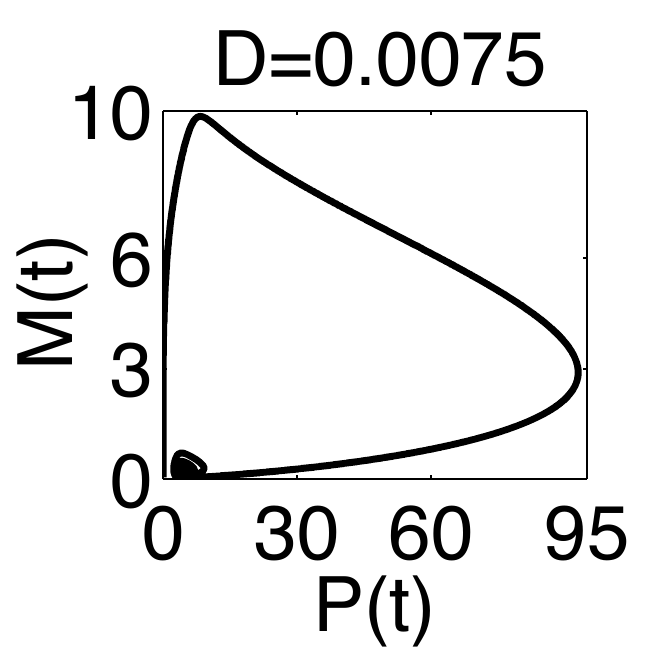} \; \; \; \; \; \; 
\includegraphics[width=0.31\textwidth]{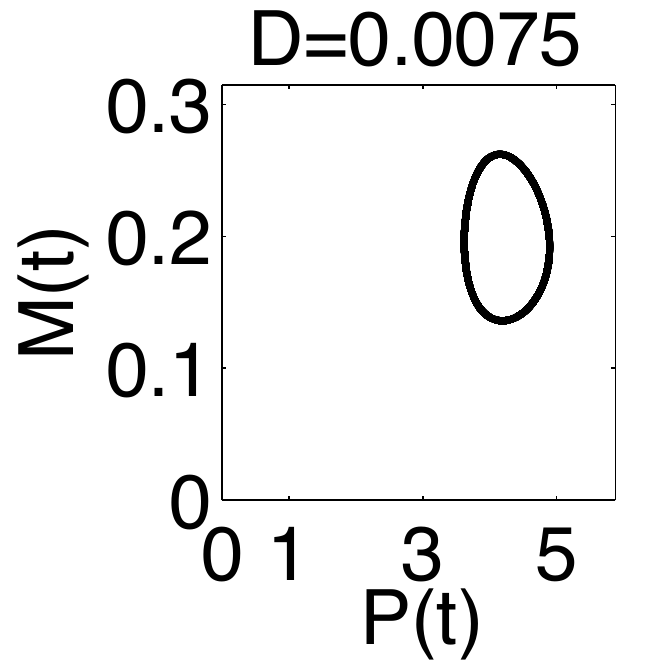}
  \caption{First two rows: plots showing the spatio-temporal evolution of mRNA level, $m(t,x)$, and protein level, $p(t,x)$, from numerical  simulations of system \eqref{dynamic_problem} with zero initial conditions, with $\ve=10^{-3}$, $D=0.0075$ and $t \in[10^4, 2 \times 10^4]$. The plots show oscillatory solutions. Bottom row:  the corresponding phase-plots, where $M(t)=\int_0^1 m(t,x) dx$ and  $P(t)=\int_0^1 p(t,x) dx$. The figure on the left is  for $t\in [0, 2 \times 10^4]$, and the figure on the right is for $t\in[10^4, 2 \times 10^4]$. These show the trajectory converging to a limit-cycle.}
  \label{fig3}
\end{figure}

 \begin{figure}[h!]
  \centering
  \includegraphics[width=0.405\textwidth]{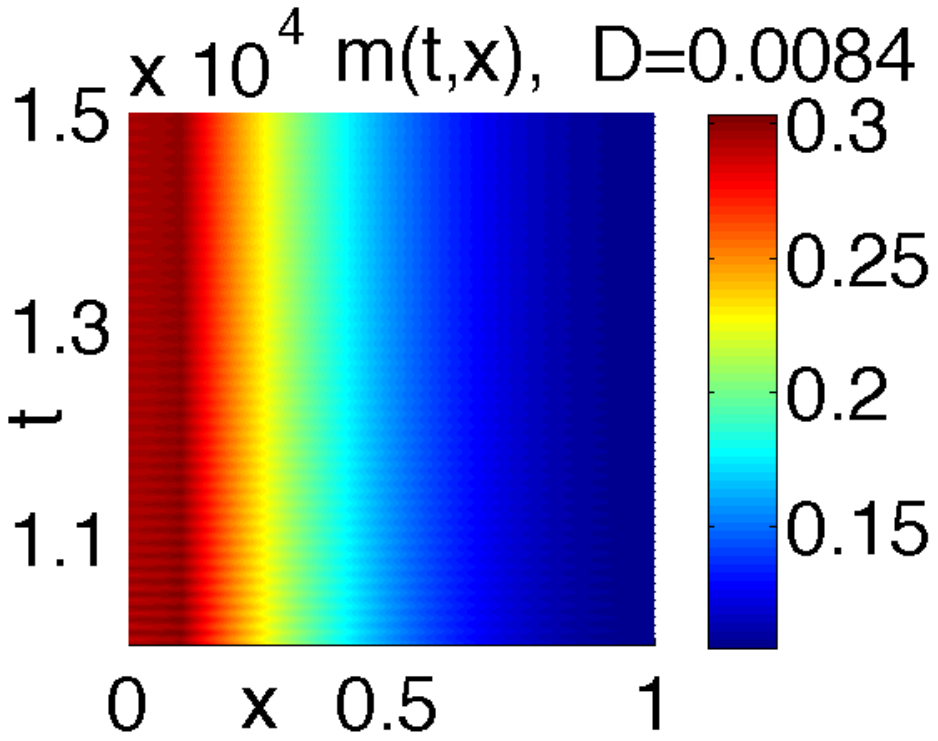}  \; \;              
\includegraphics[width=0.38\textwidth]{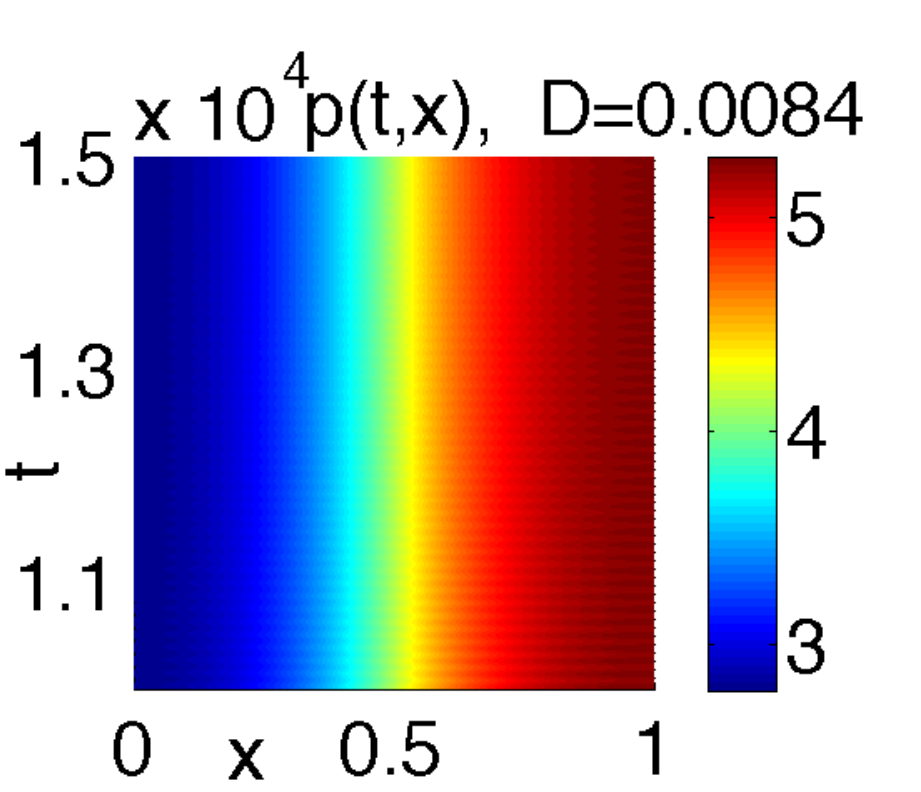}

 \includegraphics[width=0.405\textwidth]{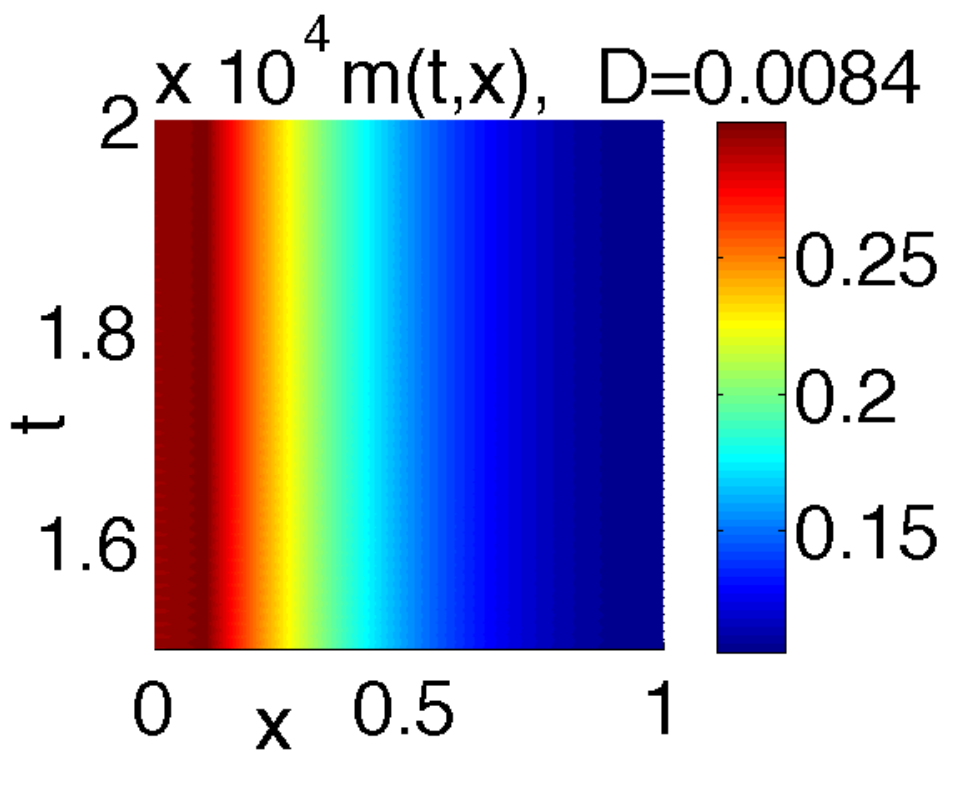}\; \;                
\includegraphics[width=0.4\textwidth]{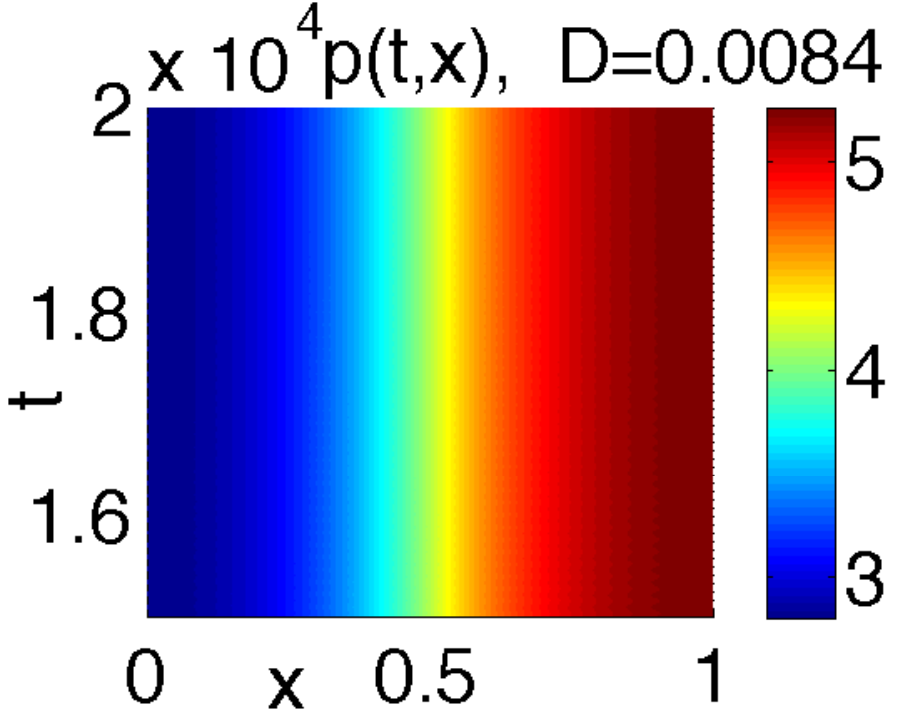}

\includegraphics[width=0.3\textwidth]{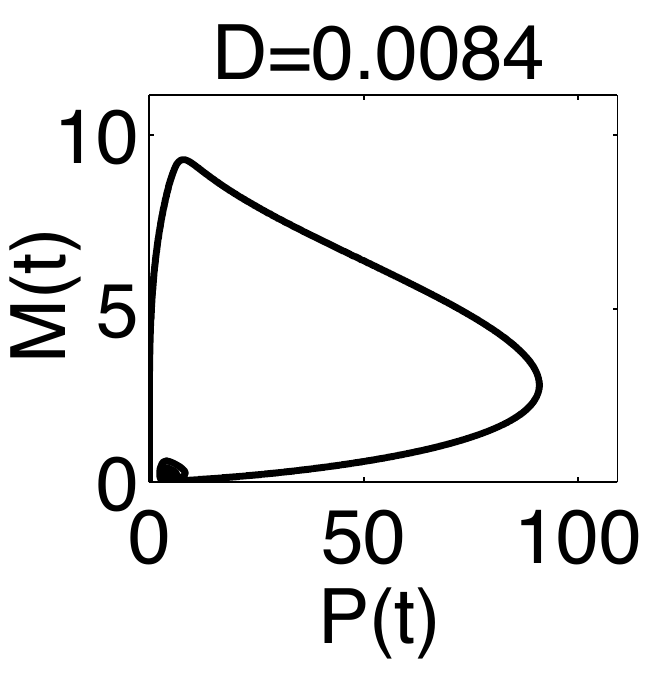}\; \; \; \; \; \; \; 
\includegraphics[width=0.3\textwidth]{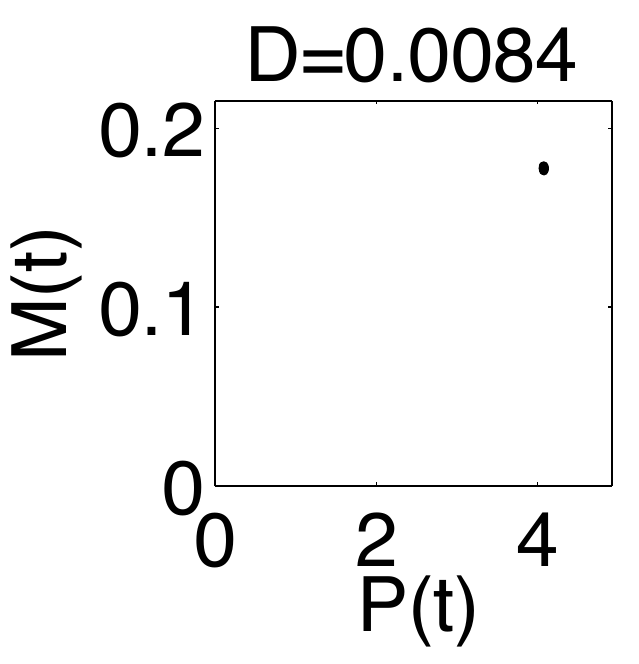}
  \caption{First two rows: plots showing the spatio-temporal evolution of mRNA level, $m(t,x)$, and protein level, $p(t,x)$, from numerical simulations of system \eqref{dynamic_problem} with zero initial conditions, with $\ve=10^{-3}$, $D=0.0084$ and $t \in[10^4, 2 \times 10^4]$. The plots show that the solutions tend to a steady-state. Bottom row: the corresponding phase-plots, where $M(t)=\int_0^1 m(t,x) dx$ and  $P(t)=\int_0^1 p(t,x) dx$. The figure on the left is for $t\in [0, 2 \times10^4]$, and the figure on the right is for $t\in[10^4, 2 \times 10^4]$. These show the trajectory converging to a fixed point, equivalent to the steady-state.}
  \label{fig4}
\end{figure}

In the following sections we shall analyse the existence and stability of a family of periodic solutions bifurcating from the stationary solution. We shall show that at both critical values of the diffusion coefficient a supercritical Hopf bifurcation occurs. 

\section{Hopf Bifurcation Analysis}
In this section we shall  prove the existence of a Hopf bifurcation for the model  \eqref{dynamic_problem} by showing that all conditions of the Hopf Bifurcation theorem are satisfied, see e.g. Crandall \& Rabinowitz \cite{Crandall}, Ize \cite{Ize} or Kielh\"ofer \cite{Kielhoefer}. In order to achieve this, we examine first the stationary solutions of \eqref{dynamic_problem}. A stationary solution  $u^\ast_\ve=(m^\ast_\ve, p^\ast_\ve)$  of the system \eqref{dynamic_problem}  satisfies the  following one-dimensional boundary-value problem: 
\begin{align}\label{stat_problem}
\begin{aligned}
&D \frac {d^2 m^\ast_\ve } { dx^2}  + \alpha_m  \, f(p^\ast_\ve ) \,  \delta^\ve_{x_M}(x) - \mu \,  m^\ast_\ve =0 \qquad \text{ in } (0,1) \; , \\
&  D \frac {d^2 p^\ast_\ve } { dx^2}  + \alpha_p \,  g(x) \, m^\ast_\ve - \mu \, p^\ast_\ve =0 \qquad \hspace{ 1.15 cm } \text{ in } (0,1)\; , \\
&\frac{d m^\ast_\ve(0)}{dx}  = \frac{dm^\ast_\ve(1)}{dx}   =0, \qquad \frac{d p^\ast_\ve(0) }{dx} = \frac{d p^\ast_\ve(1)} { dx}   =0 \; .
\end{aligned}
\end{align} 

\noindent The  operator  $ \mathcal{ \tilde A}_0=\big(D\dfrac{d^2}{dx^2} - \mu \big)$, defined on the interval $[0,1]$ and subject to the Neumann boundary conditions, 
 $$
\mathcal D( \mathcal{ \tilde A}_0) = \{ v \in H^2(0,1): \, v^\prime(0) =0, \, v^\prime(1) =0 \},
 $$ 
is invertible and solutions of the problem  \eqref{stat_problem} can be defined as
\begin{align}\label{stat_sol}
\begin{aligned}
& m^\ast_\ve(x,D) = \alpha_m \int_0^1 G_\mu(x, y) f(p^\ast_\ve(y, D))\delta_{x_M}^\ve(y) \, dy   \; ,  \\
& p^\ast_\ve(x,D) = \alpha_m\alpha_p  \int_0^1 g(z)  G_\mu(x,z) \int_0^1 G_\mu(z,y) f(p^\ast_\ve(y, D))\delta_{x_M}^\ve(y) \, dy   \, dz  \; , 
\end{aligned}
\end{align} 
where
 $$
G_{\mu}(y,x)=
\begin{cases}
\displaystyle{\frac 1{( \mu D)^{1/2} \sinh(\theta)}}  \cosh(\theta \, y) \cosh(\theta \, (1-x)) \quad \text{ for } \, 0< y<x<1 \; , \\
\\
\displaystyle{\frac 1{ (\mu D)^{1/2} \sinh(\theta)}}  \cosh(\theta \, (1-y)) \cosh(\theta \, x) \quad \text{ for } \,  0<x<y<1 \; , 
\end{cases} 
$$
with  $\theta=( {\mu}/{D})^{1/2}$, is the Green's function satisfying the boundary-value problem 
\begin{equation*}
D G_{yy}- \mu G = -\delta_x \quad \text{ in } \, (0,1), \qquad G_y(0,x)=G_y(1,x) = 0.
\end{equation*}

Due to the boundedness of  $f$ for nonnegative $p^\ast_\ve$, the continuous embedding of $H^1(0,1)$ into $C([0,1])$ and the properties of the Dirac sequence, we obtain for nonnegative solutions of   \eqref{stat_problem}  the {\it a priori} estimates  
\begin{eqnarray}\label{estim_stationary}
 \|m^\ast_\ve\|_{H^1(0,1)} \leq C\; , \,\,  \|m^\ast_\ve\|_{C([0,1])} \leq C\; ,  \, \, 
\|p^\ast_\ve\|_{H^1(0,1)} \leq C\; , \, \,   \|p^\ast_\ve\|_{H^2(0,1)} \leq C \;, 
\end{eqnarray}
with a constant $C$ independent of $\ve$. 
\\
From the second equation  in \eqref{stat_sol} we have that 
\begin{equation}\label{eq_steady_state_p}
p^\ast_\ve(x,D) = K(p^\ast_\ve(x,D))
\end{equation}
with $K(p) =\alpha_m \alpha_p  (- \mathcal{\tilde  A}_0)^{-1}\left(g(-\mathcal{\tilde  A}_0)^{-1}\big(\delta_{x_M}^\ve f(p)\big)\right)$, where  $K: C([0,1]) \to C([0,1])$ is compact, since  $(- \mathcal{\tilde  A}_0)^{-1}$ is compact.    
Consider a closed convex bounded subset $D =\{ p \in C([0,1]) : \, 0 \leq p(x) \leq C+1 \text{ for } x \in [0,1]\}$   of $C([0,1])$, where the constant $C$ is as in estimates \eqref{estim_stationary}.
The estimates \eqref{estim_stationary}  and the fact that  $K(p) > 0$ for $p \geq 0$ imply $ p - K(p)\neq 0$ for $p \in \partial D$.
Thus  Leray-Schauder degree theory, e.g. Chapter 12.B in  Smoller \cite{Smoller},  guarantees the existence  of a positive   solution  of \eqref{stat_problem}.  
The linearised  equations  \eqref{stat_problem} at the steady state  $(m^\ast_\ve, p^\ast_\ve)$ can be written
\begin{equation}\label{linear_1}
\mathcal A u =0, 
\end{equation} 
where  $u=(u_1, u_2)$ and 
 $\mathcal A = \mathcal  A_0 + \mathcal A_1$ with 
the operator $\mathcal A_0$  given as
\begin{equation}\label{operator_A0}
\mathcal A_0= \Big(D\dfrac{d^2}{dx^2} - \mu \Big) I
\end{equation}
on the interval $[0,1]$, subject to the Neumann boundary conditions, 
 \begin{equation*}
 \mathcal D(\mathcal A_0) = \{ v \in H^2(0,1)\times H^2(0,1): \, v^\prime(0) =0, \, v^\prime(1) =0 \}, 
 \end{equation*}
and the  bounded operator
 \begin{equation}\label{operator_A1}
\mathcal A_1 = \begin{pmatrix}
0 &  \alpha_m f^\prime( p^\ast_\ve(x, D))  \, \delta^\ve_{x_M}(x) \\
\alpha_p g(x) & 0
\end{pmatrix} \;  .
\end{equation}
If for a solution $u=(u_1, u_2)$ of \eqref{linear_1} we have  $u_2(x) = 0$ in $(x_M - \ve, x_M+ \ve)$,   then  $u\equiv (0,0)$ and $\mathcal A$ is invertible.
Suppose there exists a non-trivial solution of  \eqref{linear_1}  with  $u_2(x) \neq 0$ in  $ (x_M - \ve, x_M+ \ve)$.
Using the continuity of $u_2$, we can assume  $u_2(x) > 0$ in $(x_M - \ve, x_M+ \ve)$ for  small $\ve$.  Then,  the properties of $f$ and positivity of  $(-\mathcal{\tilde A}_0)$ and of the steady state  $(m^\ast_\ve, p^\ast_\ve)$ ensure  
$$
u_2(x) - \alpha_m \alpha_p (-\mathcal{\tilde A}_0)^{-1}\Big(g\, (-\mathcal{\tilde A}_0)^{-1}\big( f^\prime( p^\ast_\ve)   \delta^\ve_{x_M} u_2\big)\Big)(x) >0 \, \,  \text{ for } x \in (x_M - \ve, x_M+ \ve) \; .
$$
This last inequality implies a contradiction, since $u_2$ was a solution of \eqref{linear_1}. 
Therefore,  $\mathcal A$ is invertible for every $D \in [d_1, d_2]$. Thus for every fixed small $\ve>0$    we have a family in $D \in [d_1, d_2]$ of isolated positive stationary solutions $(m^\ast_\ve(x,D), p^\ast_\ve(x,D)) \in H^2(0,1)\times H^2(0,1)$  of \eqref{dynamic_problem}.

The {\it a priori} estimates  imply the weak convergences $m^\ast_\ve \rightharpoonup m^\ast_0 $ in $H^1(0,1)$  and $p^\ast_\ve  \rightharpoonup  p^\ast_0$  in $H^{2}(0,1)$, and, by the compact embedding of $H^1(0,1)$  in $C([0,1])$ and of $H^2(0,1)$ in $C^1([0,1])$, also strong convergence  in $C([0,1])$  and in $C^1([0,1])$ as $\ve \to 0$, respectively, where  
\begin{align}\label{stat_sol_delta}
\begin{aligned}
& m^\ast_0(x,D) = \alpha_m G_{\mu}(x, x_M) f(p^\ast_0(x_M, D)) \; ,  \\
& p^\ast_0(x,D) = \alpha_m\alpha_p  f(p^\ast_0(x_M,D))\int_{0}^1 g(y)\,  G_{\mu}(x, y) \, G_{\mu}(y, x_M) \, dy , \; 
\end{aligned}
\end{align} 
is a  solution of the model \eqref{stat_problem} with the Delta distribution $\delta_{x_M}$ instead of the Dirac sequence $\delta^\ve_{x_M}$. 
Since $x_M < l$  and  $g(y) = 0$ for $0 \leq y < l$, we  have 
$$
G_\mu (y, x_M) = \frac 1{ (\mu D)^{1/2} \sinh(\theta)}   \cosh (\theta(1-y))\cosh( \theta x_M), \quad  x_M < y < 1, \quad 
$$
where $\theta = \left(\mu/D\right)^{1/2}$ and,  using $g(y)=1$ for $l\leq y \leq 1$, we obtain 
\begin{eqnarray*}
&& p^\ast_0(x, D) =
 \frac{\alpha_m \alpha_p f( p^\ast_0(x_M, D)) }{2\mu D \sinh^2(\theta)}   \cosh(\theta  x_M) \times \\
&& \times   \Big[
 \cosh(\theta(1-x))\Big(  \cosh(\theta) y\Big|_{l}^x  - \frac 1{2\theta} \sinh(\theta(1-2y))\Big|_{l}^x \Big)_{x>l}  \\
 &&+   \cosh(\theta   x)
\Big( y \Big|_{\max\{x, l\}}^1 -\frac 1{2\theta} \sinh(2\theta(1-y))\Big|^1_{\max\{x, l\}}\Big) \Big] \; . 
\end{eqnarray*}
It can be shown numerically that the nonlinear equation 
\begin{eqnarray}\label{stationary_xM}
p^\ast_0(x_M, D) = f(p^\ast_0(x_M, D)) \frac{\alpha_p \alpha_m} 4    \frac  {\cosh^2(\theta \, x_M)}
  { \mu\,  D\, \theta \,   \sinh^2(\theta)}
\Big[ \theta +  \sinh(\theta) \Big]
\end{eqnarray}
has only one positive solution for all values of $D \in [d_1, d_2]$.

Thus, since $m^\ast_0(x, D)$ and $p^\ast_0(x, D)$  are uniquely defined by $p^\ast_0(x_M, D)$, for every $D \in [d_1, d_2]$   we have a unique positive solution of \eqref{stat_problem} with $\ve=0$.  
 Then the strong convergence of $m^\ast_\ve \to m^\ast_0$, $p^\ast_\ve \to p^\ast_0$  as $\ve \to 0$  in $C([0,1])$ and the fact that  nonnegative steady states $(m^\ast_\ve, p^\ast_\ve)$ are isolated imply the uniqueness of the positive steady state  of \eqref{dynamic_problem} for small $\ve >0$ and $D \in [d_1, d_2]$.
 
Before carrying out our analysis, to better understand the structure of the stationary solutions of \eqref{dynamic_problem} we can consider their structure under extreme values of the diffusion coefficient $D$. For very small diffusion coefficients $D \ll 1$, in the zero-order approximation we obtain 
\begin{equation*}
\begin{aligned}
&0  =  \alpha_m \, f(p^\ast_\ve )\delta^\ve_{x_M}(x) - \mu \,  m^\ast_\ve\; ,  \quad &0  =  \alpha_p \, g(x) \, m^\ast_\ve - \mu\,  p^\ast _\ve \quad \text{ in }  (0,1) \; .
\end{aligned}
\end{equation*}
 Since $g(x) =0$ for $x \in [0, l)$,  the second equation yields that $p^\ast_\ve(x,D) = 0$ in $[0,l)$ and thus   $m^\ast_\ve(x,D) = \frac{\alpha_m}{\mu} \delta^\ve_{x_M}(x)$ in $[0,1]$.  Using the fact that  $x_M \in (0,l)$ we obtain  for sufficiently small $\ve > 0$ that  $m^\ast_\ve(x,D) =0 $ for $x\in [l,1]$ and thus   $p^\ast_\ve(x,D) = 0$ in $[0,1]$. 
 Therefore for very small $D$ we have localisation of mRNA concentration around $x_M$, whereas the concentration of protein is approximately zero everywhere in $[0,1]$. 
 
For large diffusion coefficients, i.e.  $D \gg 1$ and therefore $1/D \ll 1$, we have
 \begin{equation*}
\begin{aligned}
&0 =  \frac{d^2 m_\ve^\ast} { dx^2} +  \frac{1}{D} \big(\alpha_m \, f(p^\ast_\ve )\delta^\ve_{x_M}(x) - \mu \,  m^\ast_\ve  \big) \quad \text{ in } (0,1)\; , \\
&0 =  \frac{ d^2 p_\ve^\ast}{dx^2} +  \frac{1}{D} \big(\alpha_p \, g(x) \, m^\ast_\ve - \mu\,  p^\ast_\ve  \big)\quad \hspace{1.1 cm } \text{ in } (0,1)\; , \\
& \frac{dm^\ast_\ve}{dx}(0) = \frac{d m^\ast_\ve}{dx}(1) =0 , \qquad \frac{d p^\ast_\ve} { dx} (0)= \frac{ d p^\ast_\ve} {dx} (1) = 0\; .
\end{aligned}
\end{equation*}
Thus  $m^\ast_\ve(x, D) \approx \text{const}$ and  $p^\ast_\ve (x, D) \approx  \text{const}$. 

Representative stationary solutions, calculated numerically from  \eqref{stat_sol_delta}, in the cases $D=10^{-6} \ll 1$ and $D=100 \gg 1$ can be seen in Figure~\ref{fig_steady_state}, confirming the preceding analysis. 

\begin{figure}[h!]
\begin{center}
  \includegraphics[width=0.41\textwidth]{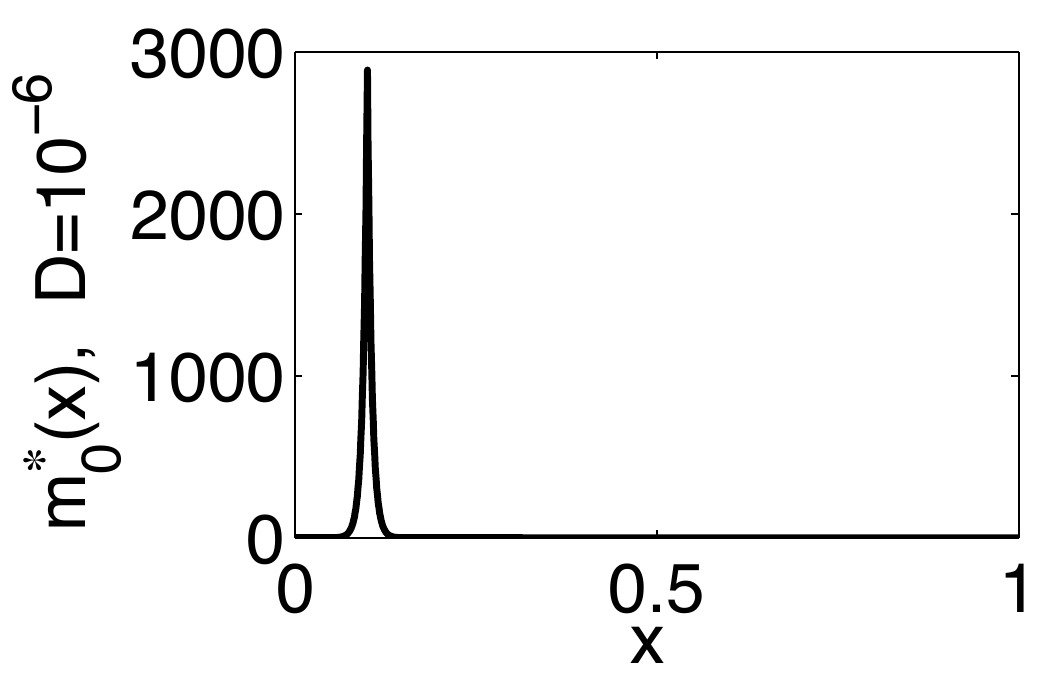} \;\;
   \includegraphics[width=0.415\textwidth]{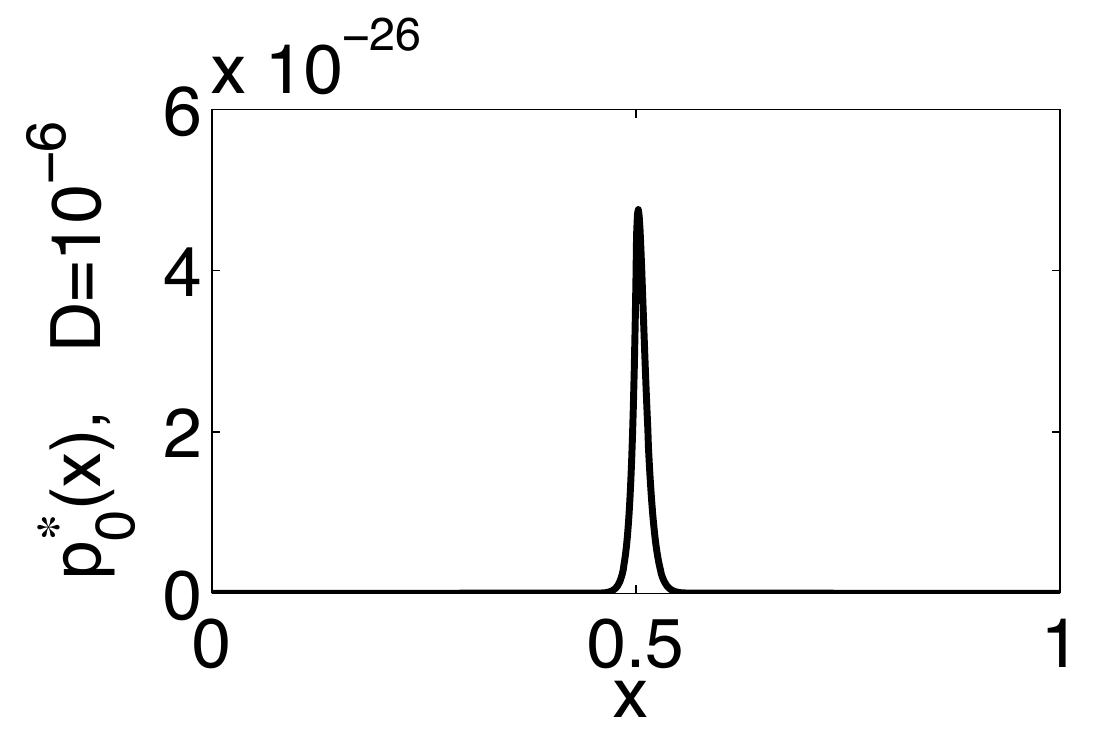}   
 \\
  \includegraphics[width=0.41\textwidth]{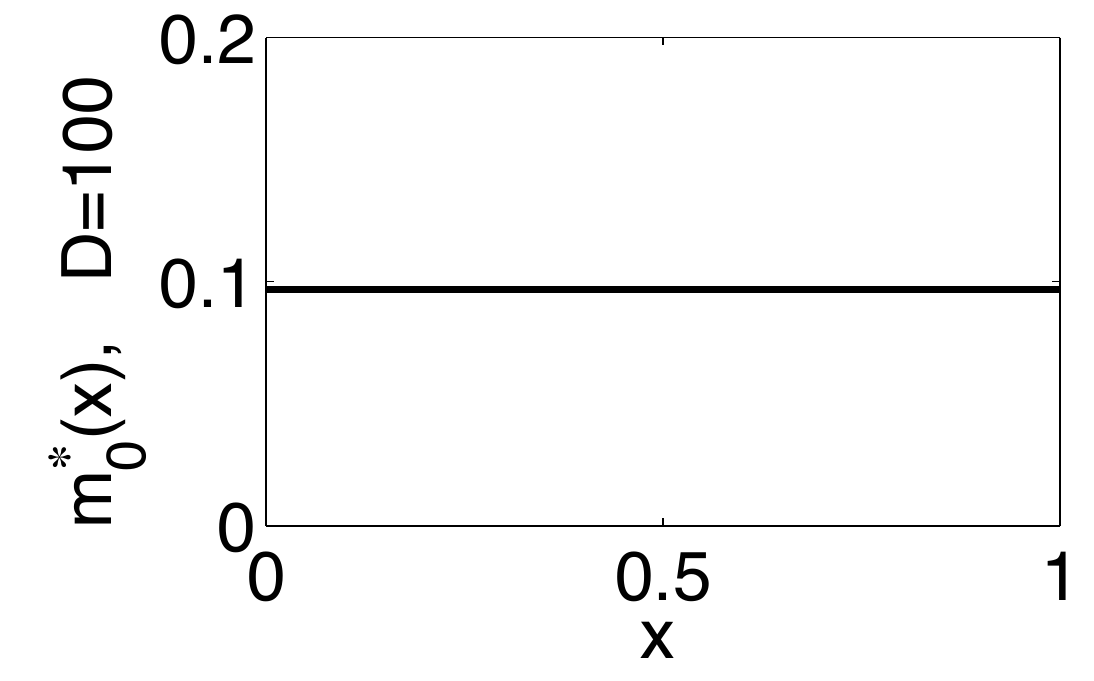} \;\;
   \includegraphics[width=0.41\textwidth]{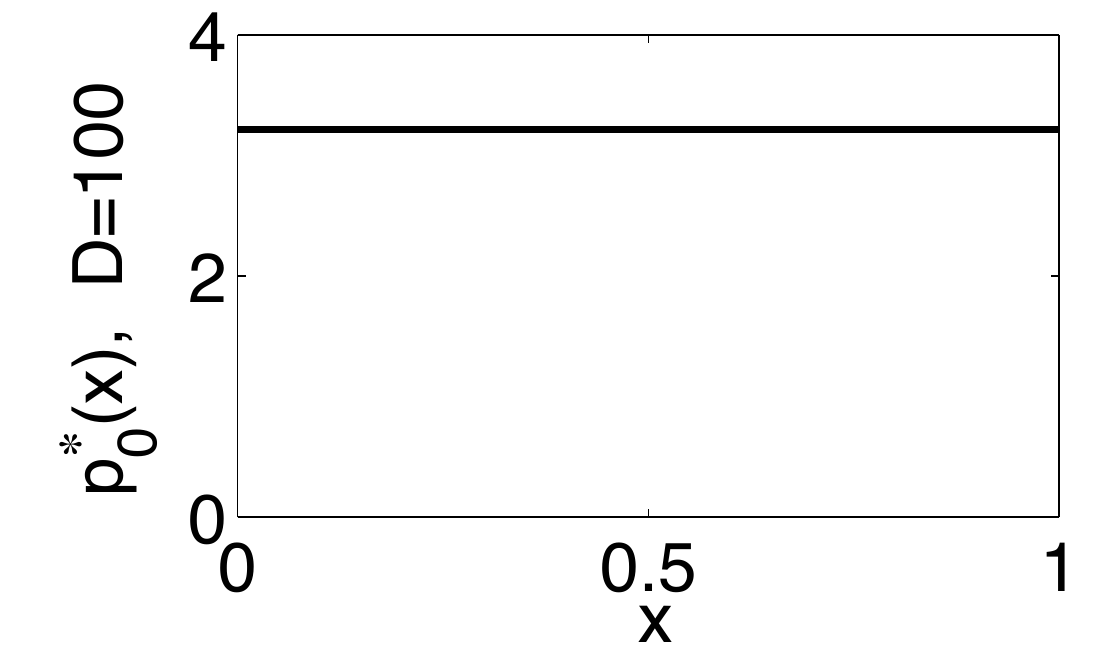} 
     \\             
\includegraphics[width=0.43\textwidth]{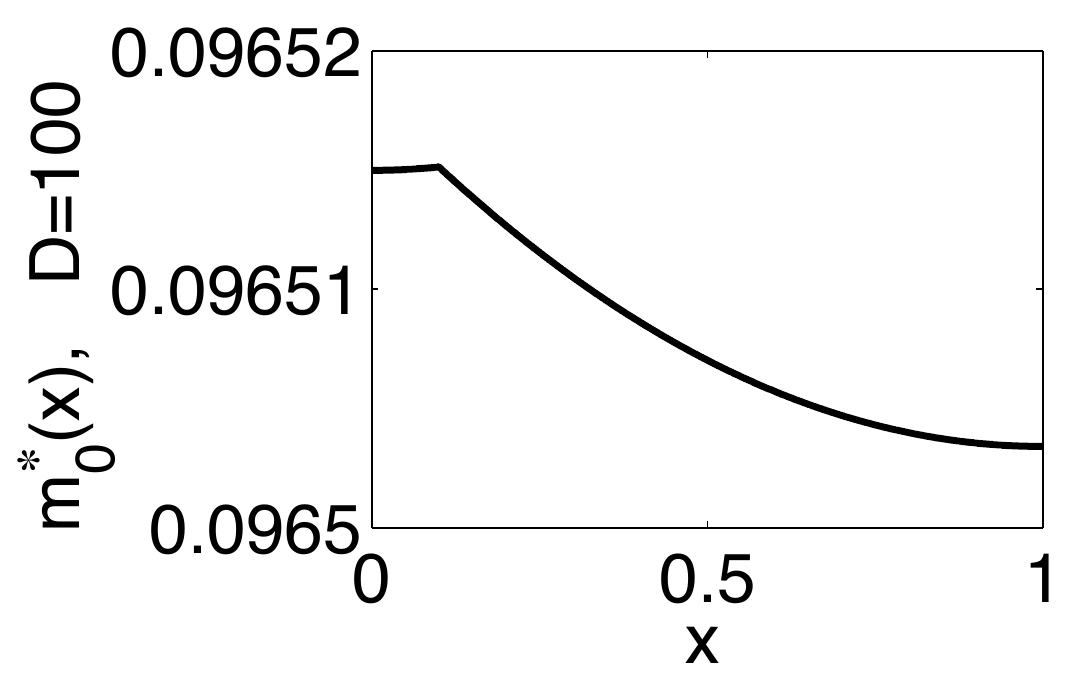}\;\;
 \includegraphics[width=0.42\textwidth]{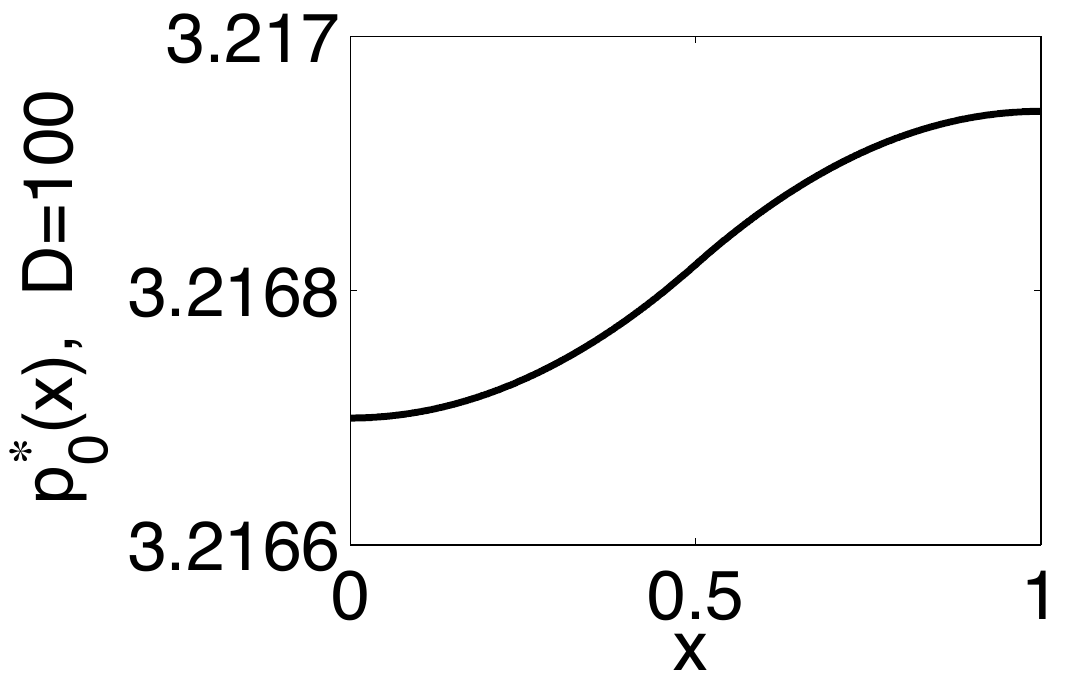} 
  \caption{Plots showing representative stationary solutions of the system \eqref{stat_sol_delta}, i.e. mRNA (left plot) and protein (right plot) steady-state concentrations, for $D=10^{-6}$ (top row) and $D=100$ (middle row, bottom row). The bottom row shows the stationary solution for $D=100$ at higher resolution.}
  \label{fig_steady_state}
  \end{center}
\end{figure}
 
Now, to study the linearized stability  of the steady-state solution  of the nonlinear model  \eqref{dynamic_problem} we shall  apply a version of Theorems 5.1.1 and 5.1.3 in Henry \cite{Henry} adapted for our situation.  We can write the system  \eqref{dynamic_problem} in the Hilbert space  $X=L^2(0,1)\otimes L^2(0,1)$ as 
\begin{equation}\label{non_lin_oper}
\partial_t u = \mathcal A_0 u + \tilde f(u), 
\end{equation}
where $u= (m, p)^T$, the operator $\mathcal A_0$ is defined in \eqref{operator_A0} and   $ \tilde f(u)= 
(\alpha_m f(p) \delta^\ve_{x_M}(x), 
\alpha_p g(x) m)^T$. 
The operator $-\mathcal A_0$ is sectorial  with  $\sigma(\mathcal A_0) \subset (-\infty, -\mu]$ and 
we can introduce interpolation spaces $X^s=\mathcal ((-\mathcal A_0)^s)$, each of which is a Hilbert subspace of $H^{2s}(0,1)\times H^{2s}(0,1)$. 
The function $\tilde f: \mathbb R^2_+ \to \mathbb R^2_+$ is smooth and admits the representation 
$$
\tilde f(y+z) = \tilde f(y) + B(y) z + r(y,z),
$$
where  the remainder satisfies the estimate 
$$
\|r(y,z)\|_{ \mathbb R^2} \leq C_\ve(y) \|z\|^2_{ \mathbb R^2}, 
$$
in a neighbourhood of any point $y \in  \mathbb R^2_+$, and 
$$
B(y) = \begin{pmatrix}
0  &  \alpha_m f^\prime(y_2)\delta^\ve_{x_M} \\
\alpha_p g(x) & 0 
\end{pmatrix} .
$$
For a positive steady state $u^\ast_\ve(x,D)=(m^\ast_\ve(x,D), p^\ast_\ve(x,D))^T$,  with $u^\ast_\ve \in H^1(0,1)\times H^2(0,1)$, we obtain that 
$B(u^\ast_\ve)$ is a bounded linear operator from $X^s$ to $X$ for each $s \in (0,1)$.
The estimate for the remainder implies 
$$
\|r(u^\ast_\ve,z)\|_X \leq  C_\ve \|z\|^2_X \leq C_\ve \|z\|^2_{X^s} = o( \|z\|_{X^s}) \to 0\quad  \text{ as }\quad  \|z\|_{X^s} \to 0,
$$
for every fixed $\ve >0$. Notice that for $s\in [1/2, 1)$, due to the properties of the Dirac sequence and the embedding of $H^1(0,1)$ into $C([0,1])$,  we have the estimates for $B$ and $r$ independent of $\ve$, i.e.
$$
\|B(u^\ast_\ve) z\|_X \leq C \|z\|_{X^s}, \quad \quad 
\|r(u^\ast_\ve,z)\|_X \leq  C \|z\|^2_{X^s} 
$$
with a constant $C$ independent of $\ve$. 

Thus, all assumptions of the Theorems 5.1.1 and 5.1.3 in Henry \cite{Henry} are satisfied and  to analyse  the linearized stability  of the stationary solution  of  the system  \eqref{dynamic_problem} we shall study  the eigenvalue problem:
\begin{align}\label{EW_problem_approx}
\begin{aligned}
&\lambda \bar m^\ve = D \bar m^\ve_{xx} + \alpha_m f^\prime( p^\ast_\ve(x, D))  \, \delta^\ve_{x_M}(x) \, \bar p^\ve - \mu  \bar m^\ve \qquad \text{ in } (0,1) \; , \\
&\lambda \bar p^\ve = D \bar p^\ve_{xx} + \alpha_p g(x) \bar m^\ve - \mu \bar p^\ve  \qquad \hspace{2.75 cm }  \text{ in } (0,1)\; , \\
& \bar m^\ve_x(0)=\bar m^\ve_x(1)=0, \, \, \bar p^\ve_x(0)=\bar p^\ve_x (1) =0, 
\end{aligned}
\end{align} 
or in operator form 
 \begin{equation}\label{EW_problem_approx_operator}
\mathcal A w^\ve = \lambda w^\ve, 
\end{equation}
where  $w^\ve=(\bar m^\ve, \bar p^\ve)^T$ and $\mathcal A = \mathcal  A_0 + \mathcal A_1$, with $ \mathcal A_1$ defined in \eqref{operator_A1}.

We can consider  $ \mathcal A $ as the perturbation of the self-adjoint operator $\mathcal A_0$  with compact resolvent by the bounded operator $\mathcal A_1$. Thus the spectrum of $ \mathcal A $ consists only of eigenvalues. Also the notion of relative boundedness \cite{Kato} can be applied  to $\mathcal A_0$ and $\mathcal A_1$.
 Let $T$ and $S$ be operators with the same domain space $\mathcal H$ such that $D(T) \subset D(S)$  and
$$
\|S u \| \leq a \| u \| + b \| T  u\| , \qquad  u \in D(T), 
$$
where  $a$, $b$  are non-negative constants. We say that $S$ is relatively bounded with respect to $T$, or simply $T$-bounded. Assume that $T$ is closed and there exists a bounded operator  $T^{-1} $, and $S$ is $T$-bounded with constants $a$, $b$ satisfying the inequality
$$
a \| T^{-1} \| + b < 1 .
$$
Then,  $T +S $ is a closed and bounded invertible operator by Theorem 1.16 of Kato \cite{Kato}.

With $\alpha_m =1$, $\alpha_p =2$, $|g(x)| \leq 1$ for all $x\in (0,1)$ and $|f^\prime (p)|= h |p^{h-1}/(1+p^h)^2| \leq 5$ ($h=5$) we have the estimate for $u \in D(\mathcal A_0)$:
\begin{equation}\label{estim_spectrum_A}
\begin{aligned}
\| \mathcal A_1 u \|_X \leq \max\big\{ \alpha_m \sup_{ x\in (0,1)} | f^\prime(p^\ast_\ve(x,D))|,  \alpha_p \sup_{x\in (0,1)}  |g(x)| \big\}( \|m\|_{L^2} + \|p\|_{H^1}) \\
\leq 5 ( \|m\|_{L^2(0,1)} + \|p\|_{H^1(0,1)}) \leq 25 \|u \|_X + 1/4 \|\mathcal A_0 u \|_X\; .
\end{aligned}
\end{equation}
Thus  we obtain that $ \mathcal A_1$ is relatively bounded   with respect to $ \mathcal A_0$ with $a=25$ and $b=1/4$.
Since  $\mathcal A_0$ is self-adjoint, we have
 $$
 \| (\mathcal A_0 - \lambda_0 I)^{-1} \| = \frac 1 { \text{dist}( \lambda_0, \sigma(\mathcal A_0))},
 $$
 and can conclude  that  $
 \mathcal A - \lambda_0 I = \mathcal A_0 + \mathcal A_1 -  \lambda_0 I
 $
 is bounded and invertible for  all $\lambda_0$ such that $\mathcal{R}e (\lambda_0) \geq 0$  and $|\lambda_0| \geq 35$ or $|\mathcal{I}m (\lambda_0)|\geq 35$ . 
Therefore we have uniform boundedness of eigenvalues $\lambda$  of the operator $\mathcal A$ with $\mathcal{R}e(\lambda) \geq 0$.

\begin{theorem}\label{Hopf_existence}
For $\ve>0$  small there exist two critical values of the parameter $D$, i.e. $D^c_{1,\ve}$ and $D^c_{2,\ve}$, for which a Hopf bifurcation occurs  in the model \eqref{dynamic_problem}. 
\end{theorem} 

\begin{proof} 
 For $\lambda \in \mathbb C$ such that $\mathcal Re (\lambda )> - \mu$  or $\mathcal Im (\lambda )\neq 0$ we can solve the first equation in the eigenvalue problem \eqref{EW_problem_approx}  for $\bar m^\ve$: 
 $$
\bar m^\ve(x) = \alpha_m (-\mathcal{\tilde A}_0)^{-1}\left(f^\prime( p^\ast_\ve(x, D)) \, \bar p^\ve(x) \,  \delta^\ve_{x_M}(x)\right)
$$
 and obtain 
 \begin{align}\label{EW_problem_approx_1}
\begin{aligned}
&\lambda \bar p^\ve = D \frac{d^2 \bar p^\ve}{dx^2}+ \alpha_p \alpha_m g(x) (-\mathcal{\tilde A}_0)^{-1}\left(f^\prime( p^\ast_\ve)\,  \bar p^\ve \, \delta^\ve_{x_M}\right) - \mu \, \bar p^\ve \; \quad \text{ in } \, (0,1)\; , \\
&\frac{d \bar p^\ve}{dx}(0)=\frac{d\bar p^\ve}{dx} (1) =0 \; .
\end{aligned}
\end{align} 
To determine the values of the parameter $D$ for which the stationary solution becomes unstable, i.e. the spectrum of $\mathcal A$ crosses the imaginary axis, we shall consider  $\lambda \in \sigma(\mathcal A)$ such that $\mathcal{R}e(\lambda) > - \mu $.   Thus $\lambda \notin \sigma(\mathcal A_0)$ and 
eigenvalue problems \eqref{EW_problem_approx} and \eqref{EW_problem_approx_1} are equivalent. 

To analyse the eigenvalue problem \eqref{EW_problem_approx_1}  further 
we  shall consider  the limit  problem obtained from \eqref{EW_problem_approx_1}  as $\ve \to 0$.
As in Dancer \cite{Dancer} we can show that for small $\ve$ the  stationary solution of \eqref{dynamic_problem} is stable  if the limit  eigenvalue problem as $\ve\to 0$
\begin{eqnarray}\label{EW_problem}
\begin{aligned}
&
\lambda \bar p = D \frac{d^2\bar p}{dx^2} - \mu \bar p + \alpha_p \alpha_m g(x)  G_{\lambda+\mu}(x, x_M)  f^\prime( p^\ast_0(x_M, D)) \bar p(x_M) \quad \text{ in } \, (0,1)\;, \\
&  \frac{d\bar p}{dx}(0)=\frac{d\bar p}{dx}(1) =0\; ,
\end{aligned}
\end{eqnarray}
has no eigenvalues with $\mathcal{R}e(\lambda) \geq 0$. 

Assume it is not true. 
Due to the upper bound for the spectrum of the operator $\mathcal A$,  shown previously, we obtain that a subsequence of eigenvalues of \eqref{EW_problem_approx_1} $\lambda_{\ve_j}$, with $\mathcal Re (\lambda_{\ve_j} )\geq 0$  and   $\ve_j \to 0$,
 converges to $\tilde \lambda$ with $\mathcal Re (\tilde  \lambda) \geq 0$ as $j \to \infty$. 
 
 For $\ve >0$, since $p^\ast_\ve \in H^2(0,1)$, $p^\ast_\ve(x,D)>0$ for $x\in [0,1]$, $D \in [d_1, d_2]$ and $f(p)$ is smooth and bounded for   nonnegative $p$, the  regularity theory implies that $(\bar m^\ve, \bar p^\ve)\in H^2(0,1)^2$ and we can normalise the solutions so that $\|\bar m^\ve\|_{L^2(0,1)} + \|\bar p^\ve\|_{L^2(0,1)}=1$.  
 From the equations in \eqref{EW_problem_approx} with $|\lambda| \leq 35$, using the normalisation and continuous embedding of $H^1(0,1)$ in $C([0,1])$,  we have estimates 
 $$
 \|\bar p^\ve\|_{H^1(0,1)} \leq C_1, \quad  \|\bar p^\ve\|_{H^2(0,1)} \leq C_2, \quad 
\|\bar m^\ve\|_{H^1(0,1)} \leq C_3\big(1+ \| \bar p^\ve \|_{H^1(0,1)}\big),
$$
where $C_1$, $C_2$ and  $C_3$   are independent of  $\ve$.
 Using compact embedding, we conclude  convergences, up to a subsequence,  
$\bar m^\ve \rightharpoonup \bar m$ weakly in $H^1(0,1)$ and strongly in $C([0,1])$ and $\bar p^\ve \rightharpoonup \bar p$ weakly  in $H^2(0,1)$  and strongly in $C^1([0,1])$.  

Additionally for $\lambda$ with $\mathcal{R}e(\lambda) \geq 0$, taking   $ \mathcal R e(\bar m^\ve) - i\mathcal I m(\bar m^\ve)$ as a test function in the first equation in  \eqref{EW_problem_approx}, using the regularity and boundedness of the stationary solution  and considering the real part of the equation  we obtain 
\begin{equation}\label{estim_ev_2}
D \left\|\frac{d \bar m^\ve}{ dx^2}\right\|^2_{L^2(0,1)}
+ \big[\mathcal Re(\lambda)+ \mu \big] \| \bar m^\ve \|^2_{L^2(0,1)} 
\leq \alpha_m
\|\bar m^\ve\|_{L^\infty(0,1)}\|\bar p^\ve\|_{L^\infty(0,1)}\; .
\end{equation}
The continuous embedding of $H^1(0,1)$ into $C([0,1])$ implies
\begin{equation}\label{cond_m_p_l}
\|\bar m^\ve\|_{L^\infty(0,1)} \leq C \|\bar p^\ve\|_{L^\infty(0,1)}\; .
\end{equation}
 Considering the strong convergence of $\bar p^{\ve_j}$ and $p^\ast_{\ve_j}$ in $C([0,1])$  and taking the limit as $j \to \infty$ in  \eqref{EW_problem_approx_1}  we obtain that $(\tilde \lambda, \bar p)$ satisfies the eigenvalue problem \eqref{EW_problem}. 
Since $\lambda$ with  $\mathcal{R}e(\lambda) \geq 0$  does not belong to $\sigma(\mathcal{\tilde  A}_0)$ we obtain  from  \eqref{EW_problem_approx_1}    that $|\bar p^\ve(x)| >0$ in $(x_M-\ve, x_M + \ve)$. Thus  due to the strong convergence of $\bar p^{\ve_j}$ in $C^1([0,1])$  we have that $\bar p(x_M) \neq 0$. Otherwise, since  for $\tilde \lambda$  with $\mathcal Re(\tilde \lambda)\geq 0$ yields $\tilde \lambda \notin \sigma(\mathcal{\tilde A}_0)$, we would obtain $ \bar p(x) = 0$ for all $x\in [0,1]$. The last result together with the estimate \eqref{cond_m_p_l} and convergence of $\bar m^\ve$ and $\bar p^\ve$ contradicts the normalisation  property $\|\bar m\|_{L^2(0,1)} + \|\bar p\|_{L^2(0,1)}=1$.
Thus
$\bar p(x) \neq 0$ in $(0,1)$ and  the problem \eqref{EW_problem}  has nontrivial solution for $\tilde \lambda$ with $\mathcal Re (\tilde  \lambda) \geq 0$. 
Therefore if there are eigenvalues of \eqref{EW_problem_approx_1} with nonnegative real part (equivalently eigenvalues with nonnegative real part of  \eqref{EW_problem_approx}) then such also exist for \eqref{EW_problem}.



Additionally using Theorem~\ref{theorem_ev} shown below or  Theorem 3 in Dancer \cite{Dancer}   we obtain that for an eigenvalue $\tilde \lambda$ of \eqref{EW_problem} with $\mathcal Re(\tilde \lambda) > - \mu$ or $\mathcal Im(\tilde \lambda)\neq 0$ there is an eigenvalue of 
\eqref{EW_problem_approx} near $\tilde \lambda$.

Therefore, if for some $D\in [d_1, d_2]$ the  problem \eqref{EW_problem} does not have eigenvalues with nonnegative real parts, then so also for the eigenvalues of \eqref{EW_problem_approx}. If for some $D\in [d_1, d_2]$ problem
 \eqref{EW_problem}   has an eigenvalue $\tilde \lambda$ with $\mathcal Re(\tilde \lambda) >0$ then for small $\ve>0$ we have in a neighbourhood of $\tilde \lambda$ eigenvalues of   \eqref{EW_problem_approx} with positive real part. 
 
We consider now the eigenvalue problem \eqref{EW_problem}.  Using the fact   $\lambda \notin\sigma(\mathcal{\tilde A}_0)$ and   applying  $(\lambda- \mathcal{\tilde A}_0)^{-1}$  in 
\eqref{EW_problem}  yields
\begin{eqnarray}\label{eigenfunction_p}
 \bar p(x) = \alpha_p \alpha_m f^\prime( p^\ast_0(x_M, D)) \bar p(x_M)  \int_{0}^1 g(y) G_{\lambda+\mu}(x,y)   G_{\lambda+\mu}(y, x_M)  dy    \; .
\end{eqnarray}
Considering  $x_M < l$, as well as   $g(x) = 0$ for $x < l$ and  $g(x)=1$ for $l\leq x \leq 1$ implies
\begin{eqnarray}\label{eigenfunction_p_2}
 \bar p(x) &=&
 \frac{\alpha_m \alpha_p f^\prime( p^\ast_0(x_M, D)) }{(\mu+\lambda) D \sinh^2(\theta_\lambda)} \bar p(x_M)  \cosh(\theta_\lambda x_M) \times  \\
&& \times  \Big[
\cosh(\theta_\lambda(1-x))\Big( \frac 12 \cosh(\theta_\lambda) y\Big|_{l}^x - \frac 1{4\theta_\lambda} \sinh(\theta_\lambda(1-2y))\Big|_{l}^x \Big)_{x>l} \nonumber \\
&& +  \cosh(\theta_\lambda x)
\Big(\frac12 y \Big|_{\max\{x, l\}}^1 -\frac 1{4\theta_\lambda} \sinh(2\theta_\lambda(1-y))\Big|^1_{\max\{x, l\}}\Big) \Big] \; ,  \nonumber
\end{eqnarray}
where $\theta_\lambda = \left((\mu + \lambda)/D\right)^{1/2}$. Then for $x=x_M< l$, where   $l=1/2$, we have 
\begin{equation}\label{eq:eigen_value1}
 \bar p(x_M) = \frac{\alpha_p\alpha_m} 4  \bar p(x_M) f^\prime( p^\ast_0(x_M,D))    \frac  {\cosh^2(\theta_\lambda \, x_M)}
  {D (\mu +\lambda) \,   \sinh^2(\theta_\lambda)}
\Big[1  + \frac 1{\theta_\lambda} \sinh(\theta_\lambda)\Big] \; .
\end{equation}

If $\bar p(x_M) =0$ and $\lambda \notin\sigma(\mathcal {\tilde A}_0)$  we have  $\bar p(x) =0$  for all $x\in (0,1)$. 
Therefore,  in the context of the analysis of the  instability of stationary solutions  of the model  \eqref{dynamic_problem}, i.e.  for  $\lambda \in \sigma(\mathcal A)$ such that $\mathcal Re(\lambda) \geq 0$, we can assume that $\bar p(x_M)\neq 0$.  
Now dividing both sides of the equation  \eqref{eq:eigen_value1} by $\bar p(x_M)$  we obtain  a nonlinear equation for eigenvalues $\lambda$  in terms of the  stationary solution  and  parameters in the model
\begin{eqnarray}\label{ev_eq}
 R(\lambda) &=&\frac{\alpha_p\alpha_m} 4 f^\prime( p^\ast_0(x_M, D))     {\cosh^2(\theta_\lambda \, x_M)} 
\big[\theta_\lambda  +  \sinh(\theta_\lambda)\big] -\theta_\lambda\,  D (\mu +\lambda) \,  \sinh^2(\theta_\lambda ) \nonumber \\
 &=& 0,  
\end{eqnarray}
where the value of the stationary solution $p^\ast_0(x_M, D)$ is defined by~\eqref{stationary_xM}.

The estimates \eqref{estim_spectrum_A} for $\mathcal A$ ensure that  $ \sigma(\mathcal A) \subset \{ \lambda \in \mathbb C: \, \mathcal Re(\lambda) \leq 35,  \, |\mathcal Im(\lambda)| \leq 35 \}$.  Since the operator $\mathcal A$ is real,  we can consider only eigenvalues with positive complex part and the corresponding  complex conjugate values will also be eigenvalues. Using Matlab and applying Newton's method  with  initial guesses in $[-5, 35]\times [0, 35]$ with step $0.001$ we solved equation~\eqref{ev_eq}  numerically  and  found two critical values of the bifurcation parameter  $D^c_1 \approx 3.117109 \times 10^{-4}$ and $D^c_2 \approx  7.884712 \times 10^{-3} $, for which we have a pair of purely imaginary eigenvalues  $\lambda_1^c \approx  17.6411537 \times 10^{-3}\, i$ and $\overline \lambda_1^c$  and   $\lambda_2^c \approx  51.2345925 \times 10^{-3}\, i$ and $\overline \lambda_2^c$ satisfying \eqref{ev_eq}.  We showed numerically  that  for  both critical values of the bifurcation  parameter     all other eigenvalues  of \eqref{EW_problem} have negative real part. From numerical simulations of the equation  \eqref{ev_eq},  we obtain  also that there are no eigenvalues   $\tilde \lambda$ of \eqref{EW_problem} with $\mathcal Re(\tilde \lambda) \geq 0$ for $D<D^c_1$  and for $D> D_2^c$, and there exist   eigenvalues  with positive real part for $D^c_1< D<D^c_2$. We also obtained  numerically  that  for $D$ such that  $D^c_1<D<D^c_2$ and close to the critical values, there is only one pair of complex conjugate eigenvalues with positive real part satisfying \eqref{ev_eq}. 

In addition to the numerical results we can  verify the simplicity of the  purely imaginary eigenvalues by  computing the derivative of $R(\lambda)$ in  \eqref{ev_eq} with respect to $\lambda$, evaluated at $\lambda^c_1$ and  $\lambda^c_2$:  
\begin{eqnarray*}
R^\prime(\lambda) &=&    
\frac{\alpha_m \alpha_p  f^\prime(p^\ast(x_M, D))  } {8(\mu+\lambda)^{1/2} D^{1/2}} 
\big[ \sinh(2\theta_\lambda x_M) x_M ( \theta_\lambda + \sinh(\theta_\lambda)) \\
&&+ \cosh^2(\theta_\lambda x_M) ( 1+ \cosh(\theta_\lambda))\big]
-\frac 12\big[ 3 D \theta_\lambda \sinh^2(\theta_\lambda) + ( \mu + \lambda) \sinh(2 \theta_\lambda)\big]\; .
\end{eqnarray*}
Simple  algebraic calculations  using Matlab (or Maple) give $R^\prime(\lambda^c_1)\approx 
-3.347 \times10^6 + 9.901 \times 10^5 i$, 
$R^\prime(\lambda^c_2)\approx 1.848 + 0.647i$ and thus  the simplicity of purely imaginary eigenvalues $\lambda^c_1$,  $\lambda^c_2$ of \eqref{EW_problem}.

To prove the transversality condition  we shall define the derivative of the eigenvalues with respect to the parameter $D$.
Differentiation of  \eqref{ev_eq} implies:
\begin{eqnarray*}
&& \frac {d\lambda } {dD}(D, \lambda)  =\Big( f^{\prime \prime} (p^\ast_0(x_M, D)) \frac{\partial  p^\ast_0(x_M, D)} {\partial D} 
\cosh ( \theta_\lambda x_M)\Big[ 1+ 
 \frac{\sinh\left(\theta_\lambda\right)}  {\theta_\lambda}  \Big]\\
 && 
  +\frac{ f^{\prime } (p^\ast_0) }{D}    \Big[   \cosh ( \theta_\lambda x_M) \Big[ 
  \frac{\theta_\lambda\cosh (\theta_\lambda)}{\sinh(\theta_\lambda)} 
     +  \frac{\cosh ( \theta_\lambda) }{2}-1 
     -    \frac { D^{\frac 12}\sinh(\theta_\lambda)  }{2 (\mu + \lambda)^{\frac 12}} 
  \Big]  \\
  && - x_M\sinh ( \theta_\lambda x_M)( \theta_\lambda + 
\sinh(\theta_\lambda))  \Big] \Big)  \Big( f^\prime ( p^\ast_0) 
 \Big[ \frac{\cosh ( \theta_\lambda x_M)}{2(\mu+ \lambda)} \Big[ 
  \frac {(\mu + \lambda)^{\frac12} \cosh ( \theta_\lambda)} { D^{\frac 12}\sinh(\theta_\lambda)}\\
  && +
  \cosh (\theta_\lambda)+ \frac {3  \sinh(\theta_\lambda) }{\theta_\lambda} +2  \Big]      -
 x_M \frac{\sinh ( \theta_\lambda x_M)}{ \mu + \lambda} \Big[ \frac { (\mu + \lambda)^{\frac 12}} { D^{\frac 12}} + \sinh(\theta_\lambda)\Big] \Big] \Big)^{-1} \; , 
\end{eqnarray*}
where $\theta_\lambda =  \big(\frac {\mu+ \lambda}  D\big)^{\frac 12}$. The derivative of the stationary solution with respect to the bifurcation parameter $D$  evaluated at $x_M$ is as follows: 
\begin{eqnarray*}
\frac{\partial p^\ast_0(x_M, D)}{\partial D} = \frac{\alpha_p \alpha_m} { 4\mu D^{\frac 32} } \,  \frac{\cosh (\theta x_M)\sinh^{-2}(\theta)} {1+ (h+1) (p^\ast_0(x_M, D))^{h}}  \Big(\cosh\big(\theta  \, x_M\big)
\Big[\frac { \mu^{\frac 12}} D  \frac { \cosh(\theta)} { \sinh (\theta)} \\
-  \frac{1}{D^{\frac 12}}
- \frac { \sinh (\theta)}{2 \mu^{\frac 12}} \nonumber  
+  \frac{ \cosh(\theta)}{2D^{\frac 12}} 
\Big] - x_M\, \sinh\big(\theta \,  x_M\big) \Big[ \frac { \mu^{\frac 12}} D + \frac{\sinh (\theta)} {D^{\frac 12} }  \Big] \Big) \; , 
\label{stat_sol_D}
\end{eqnarray*}
where $\theta = (\mu/D)^{\frac 12}$.  We evaluate the derivative $d \lambda /d D$  at the two critical parameter values  $D_1^c$ and $D_2^c$  and the corresponding  purely imaginary eigenvalues $\lambda^c_1$ and $\lambda^c_2$. The values obtained are:  
$$\dfrac{d \lambda}{d D}(D^c_1, \lambda^c_1) \approx 70.613 +47.159i  \, \text{ and  } \, \dfrac{d \lambda}{d D}(D^c_2, \lambda^c_2) \approx -0.681 + 1.696i.
$$
Thus $\mathcal{R}e\left( \frac{d\lambda }{d D}|_{D^c_j, \lambda^c_j} \right)\neq 0$ and  the eigenvalues $(\lambda^c_j(D), \overline{\lambda^c_j}(D))$ cross the imaginary axes with non-zero speed, where $j=1,2$.

Now we shall  show that all criteria for the existence of a local Hopf 
bifurcation \cite{Crandall,Ize,Kielhoefer} are satisfied by the system  \eqref{dynamic_problem} for small $\ve>0$.  Since 
for  $p\geq -\theta$, with $0<\theta <1$,   $f$ is a smooth function with respect to $p$,  we can write \eqref{dynamic_problem} as 
\begin{eqnarray*}
\partial_t \tilde u = \mathcal A \, \tilde u  +  F(\tilde u,D), 
\end{eqnarray*}
where $\tilde u  = (\tilde m, \tilde p)^T$ with  $\tilde m = m- m^\ast_\ve$, $\tilde p = p - p^\ast_\ve$,  and $F(\tilde u, D) =\alpha_m ((f(\tilde p + p^\ast_\ve) - f(p^\ast_\ve) - f^{\prime} (p^\ast_\ve) \tilde p)\delta^\ve_{x_M}(x), 0)^T$.

 We have that $\mathcal A = \mathcal A(D)$ is linear in $D$.
Since $p_\ve^\ast=p_\ve^\ast(D) $ is smooth function  for $D> 0$, we have $F \in C^2(U\times (\underline D, \overline D))$, for   $U \subset \mathbb R\times (-1, \infty)$, such that  $u^\ast_\ve=(m^\ast_\ve, p^\ast_\ve) \in U$,  and  some  $0<\underline D < d_1$ and $\overline D>d_2$. Additionally we have  $F(0,D) =0$, $\partial_{\tilde m} F(0,D) =0$,  and $\partial_{\tilde p} F(0,D) =0$ for $D\in(\underline D, \overline D)$. 

The properties of the  operator $\mathcal A_0$ and the assumption on the function $f$ ensure that 
$-\mathcal A$, as a bounded perturbation of a self-adjoint sectorial operator, is the infinitesimal generator of a strongly continuous analytic  semigroup $T(t)$ on $L^2(0,1)$ \cite{Kato,Pazy}, and $(\lambda I -\mathcal A)^{-1} $ is compact for $\lambda$ in the resolvent set of $\mathcal A$ for all  values of $D \in (\underline D, \overline D)$.

From  the analysis above and applying  Theorem~\ref{theorem_ev} shown below or  Theorem~3 in Dancer \cite{Dancer}  we can  conclude that  for small $\ve$, all eigenvalues $\lambda_\ve$ of  \eqref{EW_problem_approx}  have $\mathcal Re (\lambda_\ve) < 0$ for $D<D^c_{1}$ and $D> D^c_{2}$ and there exist  eigenvalues with $\mathcal Re(\lambda_\ve) >0$ for  $D^c_{1}< D<D^c_{2}$.
As shown before we have  $0 \notin \sigma(\mathcal A)$. 
This together with continuous dependence of eigenvalues on the parameter $D$ implies 
  that for small $\ve>0$ there are  two critical values of $D$, i.e $D^{c}_{1,\ve}$  and $D^{c}_{2,\ve}$, close to $D^c_1$ and $D^c_2$, for which we have 
a pair of purely imaginary eigenvalues for the original operator $\mathcal A$, i.e. solutions of the eigenvalue problem \eqref{EW_problem_approx}. We have also that  there  no eigenvalues of \eqref{EW_problem_approx}  with positive real part for $D=D^c_{j,\ve}$, where $j=1,2$. 

The simplicity of eigenvalues  $\lambda_{j}^c$  of  \eqref{EW_problem},  the fact that  \eqref{EW_problem} has only one  pair of   complex conjugates eigenvalues with positive real part  for  $D^c_1< D<D^c_{2}$, close to the critical values, continuous dependence of eigenvalues $\lambda$ and $\lambda_{\ve}$ on $D$ and $\ve$  together  with Theorem~\ref{theorem_ev} shown below or  Theorem 3 in   Dancer \cite{Dancer}  ensure  the simplicity of  $\lambda_{j, \ve}^c$ and $\overline\lambda_{j,\ve}^c$ 
as well as $\pm n \lambda_{j,\ve}^c \notin \sigma(\mathcal A)$, where $j=1,2$.

The transversality property of $\lambda_{j}^c$ and  the fact that $\lambda_{j,\ve} (D)$ are isolated (as zeros of an analytic function with respect to   $D$  and $\lambda$, see proof of Theorem~\ref{theorem_ev}) imply that for small $\ve > 0$   the eigenvalues $\lambda^c_{j,\ve}(D)$ and $\overline{\lambda^c_{j,\ve}}(D)$ of the problem \eqref{EW_problem_approx} cross the real line with non-zero speed as the bifurcation parameter $D$ increases, where $j=1,2$.

Then  the Hopf Bifurcation Theorem, see e.g. Crandall \& Rabinowitz \cite{Crandall},   Ize \cite{Ize},  Kielh\"ofer \cite{Kielhoefer},  ensures the existence in the neighbourhood of  $(m^\ast_\ve, p^\ast_\ve, D_{j, \ve}^c)$ of a one-parameter  family of  periodic solutions  of the nonlinear system \eqref{dynamic_problem}, bifurcating from the stationary solution  starting from 
$(m^\ast_\ve, p^\ast_\ve, D_{j, \ve}^c, T_j^0)$, where $T_j^0=2\pi/ \mathcal{I}m(\lambda_{j, \ve}^c)$ with $j=1,2$, and the period  is a continuous function of  $D$. 
\end{proof}

We shall define 
\begin{equation}\label{A_tilde}
\mathcal{\tilde A} = \begin{pmatrix}
D \frac{d^2}{dx^2} -\mu& \qquad   \alpha_m f^\prime( p^\ast_0(x, D))  \, \delta_{x_M}(x) \\
\alpha_p g(x) &\qquad  D \frac{d^2}{dx^2}-\mu
\end{pmatrix} \;  .
\end{equation}
In a manner similar to  Theorem~3 in  Dancer \cite{Dancer},  we can show for the eigenvalue problem \eqref{EW_problem_approx} the following result: 
\begin{theorem} (cf. Dancer \cite{Dancer})\label{theorem_ev}
For small $\ve>0$ we have that if $\tilde \lambda $ is an eigenvalue of \eqref{EW_problem} with $\mathcal{R}e(\tilde \lambda) > -\mu $, then there is an eigenvalue $\lambda_\ve$ of \eqref{EW_problem_approx} with $\lambda_\ve$ near $\tilde \lambda$ and $\lambda_\ve \to \tilde \lambda$ as $\ve \to \infty$.
The same result holds for $\mathcal Re(\tilde \lambda) \leq -\mu$ with $\mathcal I m(\tilde \lambda)\neq 0$.
\end{theorem} 

\begin{proof}
The proof follows the same steps as in Theorem 3 and Lemma 4 of Dancer \cite{Dancer}. 
In a manner similar to  Anselone \cite{Anselone} and Dancer \cite{Dancer}, the collective compactness of a set of operators is  used to show the result of the theorem. 
Note that for $\lambda$ with $\mathcal Im(\lambda) \neq 0$  or  for  real  $\lambda$  with $\lambda> -\mu$ the operator  $\mathcal{\tilde  A}_0 -\lambda= D\frac{d^2}{dx^2} - \mu- \lambda$ with zero Neumann boundary conditions  is  invertible. Thus 
we have that $\lambda$ is an eigenvalue of problem \eqref{EW_problem_approx} if it is an eigenvalue of 
\eqref{EW_problem_approx_1}. We denote 
$$W_\ve(\lambda) h =  \alpha_p \alpha_m g(x) \int_0^1 G_{\lambda+ \mu}(x,y)f^\prime( p^\ast_\ve)  \delta^\ve_{x_M}(y) h(y) dy - \lambda h \, \text{ for }  h \in E=C([0,1]),$$ and 
 shall prove that  for $\lambda \in T= \{ \lambda\in \mathbb C ,\,  \mathcal Re(\lambda) \geq - \mu+\vartheta, \,   |\lambda| \leq \Theta \} $,  for some $\Theta \geq  35$,   $0<\vartheta<\mu/2$, and  $\ve>0$ small, $(-\mathcal{\tilde A}_0)^{-1} W_\ve(\lambda)$
 is a collectively compact set of operators on $E$ and converges pointwise to $(-\mathcal{\tilde A}_0)^{-1} W_0(\lambda)$ as $\ve \to 0$, i.e.
 $$
 (-\mathcal{\tilde A}_0)^{-1} W_\ve(\lambda_\ve) h \to (-\mathcal{\tilde A}_0)^{-1} W_0(\lambda) h
 $$
as $\ve \to 0$, for every $h \in E$,  if $\lambda_\ve \to \lambda$ as $\ve \to 0$.
Here 
$$
W_0(\lambda)h  = \alpha_p \alpha_m g(x)  G_{\lambda+ \mu} (x, x_M)  f^\prime( p^\ast_0(x_M)) h(x_M) - \lambda h. 
$$
From the definition of $W_\ve(\lambda)$ and the properties of the function $f$ and  the Dirac sequence, as well as positivity   of the stationary solution $p^\ast_\ve$,    follows the boundedness of $W_\ve$ on $E$, i.e.
$$
\|W_\ve(\lambda) h\|_E \leq C \|h\|_E \qquad \text{ for all } \, \,  \lambda \in T, 
$$
with a constant $C$ independent of $\ve$. Then the compactness of $(-\mathcal{\tilde A}_0)^{-1}$ implies  the collective compactness of $(-\mathcal{\tilde A}_0)^{-1} W_\ve(\lambda)$ for $\lambda \in T$. 
For $h \in E$ and $\lambda_\ve \to \lambda$ as $\ve \to 0$,  using strong convergence of $p^\ast_\ve$ in $C([0,1])$, we have that
$W_\ve(\lambda_\ve)h \rightharpoonup W_0(\lambda)h$ weakly in $L^2(0,1)$. 
By the regularity of $\mathcal{\tilde A}_0$ we have that 
$ (-\mathcal{\tilde A}_0)^{-1} W_\ve(\lambda_\ve) h \rightharpoonup (-\mathcal{\tilde A}_0)^{-1} W_0(\lambda) h$ weakly in $H^2(0,1)$ and thus, by  the compact embedding of $H^2(0,1)$ into $C([0,1])$, it follows that 
$ (-\mathcal{\tilde A}_0)^{-1} W_\ve(\lambda_\ve) h \to (-\mathcal{\tilde A}_0)^{-1} W_0(\lambda) h$ strongly in $E$.  Notice that $W_\ve(\lambda)$ and $W_0(\lambda)$, for  $\mathcal{R}e(\tilde \lambda) > -\mu $ or $\mathcal{I}m(\tilde \lambda) \neq 0$, depend analytically on $\lambda$, i.e.   as  products and compositions of analytic functions in $\lambda$.

Since $ (-\mathcal {\tilde A}_0)^{-1} W_0(\lambda)$ is compact,  we have that  $I-(-\mathcal{\tilde A}_0)^{-1} W_0(\lambda) $ is a Fredholm operator with index zero, see e.g. Brezis \cite{Brezis}. Using the theory of Fredholm operators, for $\tilde \lambda$ such that   $I-(-\mathcal{\tilde A}_0)^{-1} W_0(\tilde \lambda) $ is not invertible, there exist closed subspaces $M$ and  $Y$ of $E$  such that $E= \mathcal N\oplus M$ and $E=\mathcal R \oplus Y$, where $\mathcal N = \mathcal N(I-(-\mathcal{\tilde A}_0)^{-1} W_0(\tilde \lambda)) $ and $\mathcal R=\mathcal R(I-(-\mathcal{\tilde  A}_0)^{-1} W_0(\tilde \lambda))$,  for which   $\text{dim}(Y)=\text{dim}(\mathcal N)$. 
Let $Q: E\to \mathcal R$ be the projection onto $\mathcal R$  parallel to $Y$.

Now we shall prove that $Q(I-(-\mathcal{\tilde A}_0)^{-1} W_\ve(\lambda) ):M \to \mathcal R$ is invertible  if $\lambda$ is near $\tilde \lambda$ and $\ve$ is small. 
Since $Q(I-(-\mathcal{\tilde A}_0)^{-1} W_\ve(\lambda) )= Q(I-(-\mathcal{\tilde A}_0)^{-1} W_0(\lambda) )- Q((-\mathcal{\tilde A}_0)^{-1} W_\ve(\lambda) -  (-\mathcal{\tilde A}_0)^{-1} W_0(\lambda) )$, this is a compact perturbation of a Fredholm operator of index zero and hence is a Fredholm operator of index zero, see e.g. Brezis \cite{Brezis}.  Then invertibility will follow if we show that $Q(I-(-\mathcal{\tilde A}_0)^{-1} W_\ve(\lambda) )$
has no kernel on $M$ for small  $\ve$ and $\lambda$ near $\tilde \lambda$.  We shall prove this by contradiction. 
Suppose that  for a sequence $(\ve_j, \lambda_{\ve_j})$ such that $\lambda_{\ve_j} \to \tilde \lambda$ and $\ve_j \to 0$ as $j \to \infty$, 
 there exists $z_{\ve_j} \in M$ with  $\|z_{\ve_j}\|=1$ and 
 \begin{equation}\label{kernel}
 Q(I-(-\mathcal{\tilde A}_0)^{-1} W_{\ve_j}(\lambda_{\ve_j}) ) z_{\ve_j} = 0 \; . 
 \end{equation}
 Due to the collective compactness property,  $(-\mathcal{\tilde A}_0)^{-1} W_{\ve_j}(\lambda_{\ve_j}) -  (-\mathcal{\tilde A}_0)^{-1} W_0(\tilde \lambda) $
is compact and thus  a subsequence  of $\big[(-\mathcal{\tilde A}_0)^{-1} W_{\ve_j}(\lambda_{\ve_j}) -  (-\mathcal{\tilde A}_0)^{-1} W_0(\tilde \lambda)\big] z_{\ve_j} $ converges strongly  in $E$.  The latter together with the equality  \eqref{kernel} ensures  that $Q(I-(-\mathcal{\tilde A}_0)^{-1} W_0(\tilde \lambda) )z_{\ve_j}$ converges in $E$. By  invertibility of $Q(I-(-\mathcal{\tilde A}_0)^{-1} W_0(\tilde \lambda) )|_M$ we have that $z_{\ve_j} \to z$  in $E$ as $j \to \infty$ and, since $M$ is closed,  $z\in M$ with $\|z\|=1$. 
We can rewrite 
$(-\mathcal{\tilde A}_0)^{-1} W_{\ve_j}(\lambda_{\ve_j}) z_{\ve_j} = (-\mathcal{\tilde A}_0)^{-1} W_{\ve_j}(\lambda_{\ve_j})z + 
(-\mathcal{\tilde A}_0)^{-1} W_{\ve_j}(\lambda_{\ve_j})(z_{\ve_j} - z)$. 
Using the convergence of $z_{\ve_j}$, and the uniform boundedness and  collective compactness of $(-\mathcal{\tilde A}_0)^{-1} W_{\ve_j}(\lambda_{\ve_j})$ we obtain that 
 $(-\mathcal{\tilde A}_0)^{-1} W_{\ve_j}(\lambda_{\ve_j}) z_{\ve_j}\to (-\mathcal{\tilde A}_0)^{-1} W_0(\tilde \lambda) z$  in $E$ as $j \to \infty$.
Thus we can pass to the limit in \eqref{kernel}  and obtain that $z \in \mathcal N$. This implies the contradiction since  $z\in M$ and $\|z\|=1$.  Therefore $Q(I-(-\mathcal{\tilde  A}_0)^{-1} W_\ve(\lambda) )|_M$ is invertible for $\lambda$ close to $\tilde \lambda$ and small $\ve$.

The  convergence of $Q(I-(-\mathcal{\tilde A}_0)^{-1} W_{\ve_j}(\lambda_{\ve_j}))z_{\ve_j}$  as well as  invertibility of $Q(I-(-\mathcal{\tilde  A}_0)^{-1} W_{\ve_j}(\lambda) )|_M$ and of $Q(I-(-\mathcal{\tilde  A}_0)^{-1} W_0(\tilde \lambda) )|_M$ ensure  that there exist $\kappa>0$, $j_0>0$ and $\delta>0$, independent of $\ve$  and $\lambda$, such that 
$$
\inf\{ \|Q(I-(-\mathcal{\tilde A}_0)^{-1} W_{\ve_j}(\lambda) ) z\|, \, \, z \in M, \, \, \|z \|=1,\,\,   j\geq j_0, \,\,   |\lambda - \tilde \lambda | \leq \delta \} \geq \kappa \; .
$$
Thus we have uniform boundedness of operators  $\big(Q(I-(-\mathcal{\tilde A}_0)^{-1} W_{\ve_j}(\lambda))\big)^{-1}$ mapping from $\mathcal R$  to $M$ by $\kappa^{-1}$ for $\lambda$ close to $\tilde \lambda$ and $j\geq  j_0$.
The collective compactness property of  $(-\mathcal{\tilde A}_0)^{-1} W_{\ve_j}$ and uniform boundedness of $\big(Q(I-(-\mathcal{\tilde A}_0)^{-1} W_{\ve_j}(\lambda))\big)^{-1}$ imply also that for $h \in \mathcal R$ we have  
$\big(Q(I-(-\mathcal{\tilde A}_0)^{-1} W_{\ve_j}(\lambda_{\ve_j}))\big)^{-1} h \to \big(Q(I-(-\mathcal{\tilde A}_0)^{-1} W_0(\tilde \lambda) )\big)^{-1} h$ as $j \to \infty$, $\ve_j \to 0$ and $\lambda_{\ve_j} \to \tilde \lambda$, see Theorem I.6 in  Anselone \cite{Anselone}. 

Now for $\bar p^\ve= m+k \in E$ with $k \in \mathcal N$ and $m \in M$ we rewrite   the eigenvalue equation  in \eqref{EW_problem_approx_1} as $(I-(-\mathcal{\tilde A}_0)^{-1} W_{\ve_j}(\lambda) )(m+k) =0$   and applying projection operator obtain 
$Q(I-(-\mathcal{\tilde A}_0)^{-1} W_{\ve_j}(\lambda) )m = - Q(I-(-\mathcal{\tilde A}_0)^{-1} W_{\ve_j}(\lambda) )k$.
The invertibility of $\big(Q(I-(-\mathcal {\tilde A}_0)^{-1} W_{\ve_j}(\lambda) )\big)$ on $M$ implies
$$m = -\left[Q(I-(-\mathcal {\tilde A}_0)^{-1} W_{\ve_j}(\lambda) )\right]^{-1}\left( Q(I-(-\mathcal{\tilde A}_0)^{-1} W_{\ve_j}(\lambda) )k \right)= S_{\ve_j}(\lambda) k. $$
  By arguments from above  $S_{\ve_j}(\lambda)$ are uniformly bounded with respect to $\ve_j$ and $\lambda$ for all $\lambda$ close to $\tilde \lambda$ and $j \geq j_0$ and also   for each $k \in \mathcal N$
we have $S_{\ve_j}(\lambda_{\ve_j}) k \to S_0(\tilde\lambda) k$ if $\lambda_{\ve_j} \to \tilde \lambda$ and $\ve_j \to 0$ as $j \to \infty$. 

Thus the  eigenvalue problem \eqref{EW_problem_approx_1} is reduced to  
$$Z_{\ve_j}(\lambda) k = 0  \,  \text{ for }   k \in \mathcal N, \,   \text{ where } \, 
Z_\ve(\lambda) = (I-Q)\left(I - (-\mathcal {\tilde A}_0)^{-1} W_\ve(\lambda)\left(I+ S_\ve(\lambda)\right)\right).
$$
Due to collective compactness of $(-\mathcal {\tilde A}_0)^{-1} W_{\ve_j}(\lambda)$ and convergence of $S_{\ve_j}(\lambda_{\ve_j})$ we have that  $Z_{\ve_j}(\lambda_{\ve_j}) k \to Z_0(\lambda) k $ if $\lambda_{\ve_j} \to \lambda$  and $\ve_j \to 0$ as $j \to \infty$ for each fixed $k \in \mathcal N$.
Since $\mathcal N$ is finite dimensional it follows that 
$$
\|Z_{\ve_j}(\lambda_{\ve_j})  - Z_0(\lambda) \|\to 0 \qquad \text{ as } \quad  j \to \infty\; .
$$
Now the equation for eigenvalues is given by $\det Z_{\ve_j}(\lambda)=0$. 
The analyticity of  $W_0(\lambda)$  implies that  $Z_0(\lambda)$ is  analytic in $\lambda$. 
Since $\mathcal{\tilde A}$ has compact resolvent as relative bounded perturbation of the operator $\mathcal A_0$ with compact resolvent, we have that   spectrum  of $\mathcal{\tilde A}$ is discrete and consists of eigenvalues, see e.g. Kato \cite{Kato}. 
This implies that   $Z_0(\lambda)$ is invertible for some $\lambda \in T$ and, thus  $\det Z_0(\lambda)$  does not vanish identically on $T$ and  its zeros are isolated, i.e. $\tilde \lambda$ is an isolated zero of the  analytic   function $\det Z_0(\lambda)$. 
Therefore the topological degree \cite{Nirenberg} of $\det Z_0(\lambda)$ is positive  in the neighbourhood of $\tilde \lambda$.
Using the uniform convergence  $\det Z_{\ve_j}(\lambda_{\ve_j}) \to \det Z_0(\lambda)$ as $j \to \infty$ and homotopy invariance of the topological degree, this implies that  the degree of $\det Z_{\ve_j}(\lambda_{\ve_j})$ is equal to the degree of  $\det Z_0(\lambda)$ and is positive in the neighbourhood of $\tilde \lambda$ and small $\ve_j$.  Thus for small $\ve$ it follows that $\det Z_\ve(\lambda_\ve)$ has a solution near $\tilde \lambda$ and hence $\mathcal A$ has an eigenvalue near $\tilde \lambda$. 
 
Since $W_\ve(\lambda)$ and $W_0(\lambda)$ are analytic in $\lambda$, the sum of multiplicities of the eigenvalues of $ \mathcal A$ near $\tilde \lambda$ is equal to the multiplicity of the eigenvalue $\tilde \lambda$ of \eqref{EW_problem} \cite{Dancer}.
\end{proof}


\section{Stability of the Hopf Bifurcation}
In this section we shall analyse the stability of periodic orbits bifurcating from the stationary solution at the two critical values of the bifurcation parameter, $D_{1, \ve}^c$ and $D_{2, \ve}^c$. 
To show the stability of the Hopf bifurcation we shall use techniques from weakly nonlinear analysis.
 The method of nonlinear analysis distinguishes between  fast  and  slow  time scales in   the dynamics of  solutions near the steady state. The fast time scale corresponds to the interval of time where the linearised stability analysis is valid, whereas at  the slow time scale  the effects of the nonlinear terms become important. 
  
 \begin{theorem} 
At both critical values of the bifurcation parameter,  $D^c_{1, \ve}$ and $D^c_{2, \ve}$, a supercritical Hopf bifurcation occurs  in the system \eqref{dynamic_problem} and  the family  of periodic orbits bifurcating from the stationary solution at each Hopf bifurcation point is stable. 
\end{theorem} 

\begin{proof}
We consider a perturbation analysis in the neighbourhood of the critical parameter value $D = D^c_{j,\ve}+ \delta^2 \nu+ \cdots $, where $\nu= \pm 1$, and the corresponding  small perturbation  of the critical eigenvalues $\lambda_{j,\ve}(D)=  \lambda^c_{j,\ve} + \dfrac{\partial{\lambda_{j,\ve}}}{\partial D} \delta^2 \nu +  \cdots$, where $\delta > 0$ is a small parameter and $j=1,2$. 
As  solutions of  \eqref{dynamic_problem} near the bifurcation points are of the form $e^{\lambda_{j,\ve} t} \xi(x) + c.c.+ u^\ast_\ve(x, D) \approx e^{\lambda^c_{j,\ve} t + \frac{\partial{\lambda_{j,\ve}}}{\partial D} \nu \delta^2 t } \xi(x) + c.c. + u^\ast_\ve(x,D^c_{j,\ve}+ \delta^2 \nu + \cdots)$,  where $u^\ast_\ve=(m^\ast_\ve , p^\ast_\ve)$ is the stationary solution,  $\xi$ is the corresponding eigenvector and  $c.c.$ stands  for the complex conjugate terms, we obtain that the amplitude depends on two time scales - the fast time scale $t$ and the slow time scale $T= \delta^2 t$.  For small $\delta>0$ we shall regard $t$ and $T$ as being independent. 
 Thus we consider the solution of the nonlinear system  \eqref{dynamic_problem}  near the steady state   in the form 
\begin{eqnarray}\label{Ansatz}
\begin{aligned}
m(t,T, x) =  m^\ast_\ve(x, D) + \delta m_1(t,T, x) + \delta^2 m_2(t,T, x) + \delta^3 m_3(t,T,x) + O(\delta^4), \;\;\;\\
p(t,T, x) = p^\ast_\ve(x, D) + \delta p_1(t,T, x) + \delta^2 p_2(t,T,x) + \delta^3 p_3(t,T,x) + O(\delta^4),\quad \; \; \;\;
\end{aligned}
\end{eqnarray}
and $D = D^c_{j,\ve}+ \delta^2 \nu  $, where $\nu= \pm 1$ and $j=1,2$.
We shall use  the Ansatz \eqref{Ansatz} in equations \eqref{dynamic_problem}  and compare the terms of the same order in $\delta$. Using the regularity of $f(p)$ with respect to $p$ and of the stationary solution $u^\ast_\ve$ with respect to $D$, we shall apply Taylor series expansion to $f$  and $u^\ast_\ve$ about  $u^\ast_\ve(x,D^c_{j,\ve})$ and $D^c_{j,\ve}$, with $j=1,2$.
For $\delta$  we have: 
\begin{eqnarray}\label{eq:m1_p1}
\begin{cases}
\partial_t m_1 = D^c_{j,\ve} \partial^2_x m_1 - \mu  m_1 + \alpha_m f^\prime(p^\ast_\ve(x, D^c_{j, \ve}))  \, \delta^\ve_{x_M}(x) \, p_1, \\
\partial_t p_1 = D^c_{j,\ve} \partial^2_x p_1  - \mu p_1 + \alpha_p g(x) m_1 , \\
 \partial_x m_1(t,0)= \partial_x m_1 (t,1)=0, \, \, \partial_x p_1(t,0)
=\partial_x p_1(t,1) =0.
\end{cases}
\end{eqnarray}
The linearity of the equations as well as the fact that the dynamics near the bifurcation point is defined  by the largest eigenvalues  $\pm\lambda^c_{j, \ve}$, imply that we can consider  $m_1$ and $p_1$  in the form: 
\begin{eqnarray*}
m_1(t,T,x)= A(T) e^{\lambda^c_{j, \ve} t} \xi_1(x) + \bar A(T) e^{\bar \lambda^c_{j, \ve} t}\bar  \xi_1(x), \\
p_1(t,T,x)= A(T) e^{\lambda^c_{j, \ve} t} \xi_2(x) + \bar A(T) e^{\bar \lambda^c_{j, \ve} t}\bar  \xi_2(x),
\end{eqnarray*}
where $\xi=(\xi_1,\xi_2)$ is a  solution of the eigenvalue problem \eqref{EW_problem_approx} for  $\lambda^c_{j, \ve}= i\omega^c_{j}$, with $j=1,2$,
and $\bar \xi$ is the complex conjugate of $\xi$  (we shall omit the dependence on $\ve$ to simplify the presentation).  
For $\delta^2$, we obtain equations for $m_2$ and $p_2$:
\begin{eqnarray}\label{eq:m2_p2}
\begin{cases}
\partial_t m_2 = D^c_{j, \ve} \partial^2_x m_2 - \mu  m_2 + \alpha_m\Big[ f^\prime(p^\ast_\ve(x, D^c_{j, \ve}))   p_2+  f^{\prime\prime}(p^\ast_\ve(x, D^c_{j, \ve}))  \dfrac {p_1^2} 2  \Big]  \delta^\ve_{x_M}, \\
\partial_t p_2 = D^c_{j, \ve} \partial^2_x p_2  - \mu p_2 + \alpha_p g(x) m_2 , \\
\partial_x m_2 (t,0)= \partial_x m_2 (t,1)=0, \, \quad \partial_x p_2 (t,0)=\partial_x p_2 (t,1) =0.
\end{cases}
\end{eqnarray}
Then, due to the quadratic term in \eqref{eq:m2_p2} comprising $p_1$, the functions $m_2$ and $p_2$ are of the form:
\begin{eqnarray}\label{m_p_2}
\begin{aligned}
m_2(t,T,x)= A(T)^2 e^{2i\omega^c_j  t} w_1(x)+   c.c.
+ |A(T)|^2 \tilde w_1(x) ,  \\
p_2(t,T,x)= 
A(T)^2 e^{2i\omega^c_j t} w_2(x)  + c.c. 
+ |A(T)|^2 \tilde w_2(x) .
\end{aligned}
\end{eqnarray}
Using \eqref{m_p_2}  in equations \eqref{eq:m2_p2}, we obtain   for the terms with  $e^{2 i \omega^c_j   t}$:  
\begin{eqnarray}\label{w_problem}
\begin{cases}
2i \omega^c_j    w_1 = D^c_{j, \ve} \dfrac{d^2 w_1}{d x^2} - \mu  w_1 + \alpha_m \Big[f^\prime(p^\ast_\ve(x, D^c_{j, \ve})) w_2 + f^{\prime\prime}(p^\ast_\ve(x, D^c_{j, \ve}))  \dfrac{ \xi_2^2 }2 \Big] \delta^\ve_{x_M} , \\
2 i \omega^c_j    w_2 = D^c_{j, \ve} \dfrac{d^2 w_2}{d x^2}  - \mu \, w_2 + \alpha_p\,  g(x) \, w_1 , \\
 \dfrac{d w_1} {d x} (0)=\dfrac{d w_1} {d x} (1)=0, \, \, \quad \dfrac{d w_2}{d x} (0)=\dfrac{d  w_2}{d x} (1) =0\; , 
\end{cases}
\end{eqnarray}
as well as the corresponding complex conjugate problem for the terms with $e^{-2i \omega^c_j t}$, where $j=1,2$. 
Considering the terms for   $e^{0}$, we obtain that $\tilde w= (\tilde w_1, \tilde w_2)$ solves 
\begin{eqnarray}\label{tilde_w_proble}
\begin{cases}
0 = D^c_{j, \ve} \dfrac{d^2 \tilde w_1}{dx^2} - \mu  \tilde  w_1 + \alpha_m f^\prime(p^\ast_\ve(x, D^c_{j, \ve}))  \delta^\ve_{x_M}\tilde w_2+\alpha_m  f^{\prime\prime}(p^\ast_\ve(x, D^c_{j, \ve}))  \delta^\ve_{x_M} |\xi_2|^2 \; , \\
0 = D^c_{j, \ve} \dfrac{d^2 \tilde w_2}{dx^2}  - \mu  \tilde w_2 + \alpha_p \, g(x) \,  \tilde w_1 , \\
 \dfrac {d \tilde w_1}{dx} (0)=\dfrac {d  \tilde w_1}{dx}(1)=0, \, \quad \dfrac {d  \tilde w_2}{dx}(0)=\dfrac {d \tilde w_2}{dx}  (1) =0.
\end{cases}
\end{eqnarray}
Considering the terms of order $\delta^3$, we obtain the following  equations for $m_3$ and $p_3$:
\begin{eqnarray*}
\begin{cases}
\partial_t m_3 + \partial_T m_1&= D^c_{j, \ve} \partial^2_x m_3 - \mu \,  m_3  +  \alpha_m f^\prime(p^\ast_\ve(D^c_{j, \ve})) \delta^\ve_{x_M}  p_3 \\ & + \;\;\;\;\;\; \nu \left[ \partial^2_x m_1 +  \alpha_m \, f^{\prime\prime}(p^\ast_\ve(D^c_{j, \ve}))  \partial_D p^\ast_\ve(D^c_{j, \ve}) \, \delta^\ve_{x_M} p_1\right]\\
&
+ \; \alpha_m f^{\prime\prime}(p^\ast_\ve(D^c_{j, \ve}))   \delta^\ve_{x_M}  p_1 p_2
 + \frac 16 f^{\prime\prime\prime}(p^\ast_\ve(D^c_{j, \ve}))   \delta^\ve_{x_M} p_1^3\;  , \\
\partial_t p_3 + \partial_T p_1&= D^c_{j, \ve}\partial^2_x   p_3 - \mu \, p_3 + \alpha_p \, g(x)\,  m_3
+ \nu  {\partial^2_x p_1}   \; , \\
&\partial_x m_3 (t,0)= \partial_x m_3 (t,1)=0, \, \, \partial_x p_3 (t,0)=\partial_x p_3  (t,1) =0 \; .
\end{cases}
\end{eqnarray*}
Thus we obtain that $m_3(t,T,x)$ and $p_3(t,T,x)$ have the form 
\begin{eqnarray*}
m_3(t,T,x)&=A(T)^3 e^{3 i \omega^c_j  t} q_1(x) 
+ A(T)^2 e^{2 i \omega^c_j  t} s_1(x) 
+   A(T) e^{i \omega^c_j  t}\tilde  \xi_1(x) 
\\ 
& +
A(T) |A(T)|^2  e^{ i \omega^c_j  t} r_1(x) 
+ c.c. + |A(T)|^2 \tilde u_1(x)\; , \\
p_3(t,T,x) &= A(T)^3 e^{3 i\omega^c_j t} q_2(x) 
 + A(T)^2 e^{2i \omega^c_j   t} s_2(x) 
+  A(T) e^{i \omega^c_j  t} \tilde \xi_2(x) 
\\& +
A(T) |A(T)|^2 e^{i \omega^c_j  t} r_2(x) 
+ c.c. + |A(T)|^2 \tilde u_2(x)
  \; .
\end{eqnarray*}
Combining the terms in front of  $e^{i \omega^c_j   t}$,  we obtain equations: 
\begin{eqnarray}\label{eq:m3_p3}
\begin{cases}
  i\omega^c_j \big[A(T) \tilde \xi_1 +A(T) |A(T)|^2 r_1\big] = \big(D^c_{j, \ve}\dfrac{d^2}{dx^2}  - \mu\big) \big[A(T)\tilde \xi_1 + A(T) |A(T)|^2 \,  r_1\big] \\
 \hspace{ 5 cm }  +  \alpha_m f^{\prime } (p^\ast_\ve) \delta^\ve_{x_M} \big[A(T) \, \tilde \xi_2  + A(T) |A(T)|^2  r_2 \big]\\
  \hspace{ 2. cm } 
-\partial_T A(T) \xi_1+  A(T)\nu \Big[ \frac{d^2 \xi_1}{dx^2}  
+ \alpha_m f^{\prime \prime }  (p^\ast_\ve) \delta^\ve_{x_M}   \partial_D p^\ast_\ve\xi_2 \Big] \\
 \hspace{ 2.1 cm } 
+  A(T) |A(T)|^2\alpha_m \delta^\ve_{x_M}\Big[ f^{\prime \prime }  (p^\ast_\ve)  \big(w_2 \bar \xi_2 + \tilde w_2 \xi_2\big)
+  \frac 12 f^{\prime \prime \prime}  (p^\ast_\ve)  \xi^2_2   \bar \xi_2 \Big],
\\
 i\omega^c_j \big[  A(T)\tilde \xi_2 + A(T) |A(T)|^2 r_2\big] = 
 (D^c_{j,\ve} \dfrac{d^2}{dx^2}  - \mu) \big[A(T)\tilde \xi_2 + A(T) |A(T)|^2  r_2\big]  \\
 \hspace{ 1.9 cm }-  \partial_T A(T) \xi_2 +  A(T) \nu \frac{d^2 \xi_2}{dx^2}    +
 \alpha_p\,  g(x)  \big[A(T)\tilde \xi_1 + A(T) |A(T)|^2  r_1 \big] \; .
 \end{cases}
 \end{eqnarray}
Similar equations are obtained for $e^{- i\omega^c_j  t}$ with corresponding complex conjugate terms.
  Since $i\omega^c_j$ is an eigenvalue of $\mathcal A$,   by the Fredholm alternative,
the system \eqref{eq:m3_p3}  together with zero-flux boundary conditions has a solution 
if and only if
 \begin{eqnarray*}
&& \partial_T A(T)\big(\langle \xi_1, \xi_1^\ast\rangle  +  \langle \xi_2, \xi_2^\ast\rangle \big) \\
&&
- \nu A(T) \Big[ \langle\frac{d^2 \xi_1}{dx^2}, \xi^\ast_1 \rangle + \alpha_m
\langle f^{\prime \prime }  (p^\ast_\ve)   \partial_D p^\ast_\ve \delta^\ve_{x_M} \,  \xi_2  , \xi_1^\ast \rangle  +  \langle\frac{d^2 \xi_2}{dx^2} , \xi^\ast_2 \rangle  \Big] \\
&& - A(T) |A(T)|^2   \alpha_m
\Big[ \langle f^{\prime \prime }  (p^\ast_\ve) \delta^\ve_{x_M}  (w_2 \bar \xi_2 +  \tilde w_2 \xi_2), \xi_1^\ast \rangle  
+
\frac 12 \langle f^{\prime \prime \prime}  (p^\ast_\ve) \delta^\ve_{x_M} \xi_2|\xi_2|^2, 
\xi_1^\ast \rangle \Big]=0,
\end{eqnarray*}
where $\xi^\ast$ is the eigenvector for $\lambda =-i\omega^c_j$
 of the formal adjoint operator  $\mathcal A^\ast$:
\begin{equation}\label{EW_problem_omega_adjont}
\begin{cases}
-i\omega^c_j \xi_1^\ast = D^c_{j,\ve} \dfrac{d^2}{dx^2}\xi_1^\ast - \mu  \, \xi_1^\ast+ \alpha_p \, g(x) \, \xi_2^\ast \; , \\
- i\omega^c_j \xi_2^\ast =  D^c_{j,\ve} \dfrac{d^2}{dx^2} \xi_2^\ast - \mu \xi_2^\ast +
 \alpha_m f^\prime( p^\ast_\ve(x, D^c_{j,\ve})) \, \delta^\ve_{x_M} \, \xi_1^\ast \; 
, \\
 \dfrac d {dx} \xi^\ast_1(0)=\dfrac d{dx} \xi^\ast_1(1)=0, \, \quad \dfrac d{dx} \xi^\ast_2(0)=\dfrac d{dx} \xi^\ast_2 (1) =0 \; .
\end{cases} 
\end{equation}
By choosing $\xi^\ast$ in such a way that
$
 \langle \xi, \xi^\ast\rangle= \langle \xi_1, \xi_1^\ast\rangle + \langle \xi_2, \xi_2^\ast\rangle=1,
 $
 we obtain the equation for the amplitude  
  \begin{eqnarray*}
&& \partial_T A(T) =  a_{j, \ve} \nu A(T) + b_{j, \ve} A(T) |A(T)|^2, 
\end{eqnarray*}
where 
\begin{eqnarray*}
a_{j, \ve}&=&  \langle\frac{d^2}{dx^2} \xi_1, \xi^\ast_1 \rangle + \alpha_m
\langle f^{\prime \prime }  (p^\ast_\ve(x, D^c_{j,\ve}))   \partial_D p^\ast_\ve(x,D^c_{j,\ve}) \delta^\ve_{x_M} \,  \xi_2 , \xi_1^\ast \rangle  +  \langle\frac{d^2}{dx^2} \xi_2, \xi^\ast_2 \rangle  , \\
 b_{j,\ve}&=&  \alpha_m
 \langle f^{\prime \prime }  (p^\ast_\ve(x, D^c_{j,\ve})) \delta^\ve_{x_M} (w_2 \bar \xi_2 +  \tilde w_2 \xi_2)
 + \frac 12 f^{\prime \prime \prime}  (p^\ast_\ve(x, D^c_{j,\ve}))\delta^\ve_{x_M}\, \xi_2 |\xi_2|^2 , 
\xi_1^\ast \rangle   \; .
 \end{eqnarray*}
We can calculate the values of $b_{j, \ve}$ for $\ve=0$  which then, using the continuity with respect to $\ve$ and convergence of $b_{j,\ve}$ to $b_{j,0}$ as $\ve \to 0$, ensured by the strong convergence in  $C([0,1])$
of $p^\ast_\ve$, $w_2=w_2^\ve$, $\tilde w_2=\tilde w^\ve_2$, $\xi_2=\xi_2^\ve$ and $\xi_1^\ast=\xi_1^{\ast, \ve}$ as $\ve \to 0$, will provide the information on the type of the Hopf bifurcation and the stability of periodic orbits  for the original model \eqref{dynamic_problem} with small $\ve>0$. Since  eigenfunctions are defined {\it modulo} a constant, we can choose $\xi_2^0(x_M)=1$ and  obtain 
\begin{eqnarray*}
b_{j,0}
&=& \alpha_m \Big[f^{\prime \prime} (p^\ast(x_M, D^c_j)) 
 \big(w_2^0(x_M) + \tilde w_2^0(x_M)  \big)+
\frac{1} 2 f^{\prime \prime \prime} (p^\ast(x_M, D^c_j))\Big]  \overline{ \xi^{\ast, 0}_1}(x_M) \; , 
\end{eqnarray*}
where $j=1,2$,    $\xi^{\ast,0} = (\xi^{\ast,0}_1, \xi^{\ast, 0}_2)$ is  a solution  of the formal adjoint eigenvalue problem with  $\ve=0$,  and  $w^0=(w_1^0, w_2^0)$ and  $\tilde w^0=(\tilde w_1^0, \tilde w_2^0)$ are solutions of  \eqref{w_problem}
 and  \eqref{tilde_w_proble} for $\ve =0$, respectively.
Since $2\lambda^c_j \notin \sigma(\mathcal{\tilde A})$, for $j=1,2$,   and  $0 \notin \sigma(\mathcal{\tilde A})$, with $\tilde A$ defined in \eqref{A_tilde},  there exist unique solutions of the problems   \eqref{w_problem} and  \eqref{tilde_w_proble} for $\ve = 0$. Using  $2\lambda_j^c \notin \sigma(\mathcal{\tilde A}_0)$ and   $\xi^0_2(x_M)=1$  we can compute
\begin{eqnarray*}
&& w_2^0(x_M) = \frac {\alpha_p \alpha_m} 2  f^{\prime\prime}( p^\ast_0(x_M, D^c_j))  
 \mathcal G_1(x_M) \times 
  \Big( 1-   \alpha_p \alpha_m  f^\prime( p^\ast_0(x_M, D^c_j))  \mathcal G_1(x_M)  \Big)^{-1}  , 
\end{eqnarray*}
where 
\begin{eqnarray*}
 \mathcal G_1(x_M) = \frac { \cosh^2(\theta_{2\lambda^c_j} x_M)}  { 4(\mu+2\lambda^c_j) D^c_j \sinh^2(\theta_{2\lambda^c_j})}   \left[1 + \frac 1{ \theta_{2\lambda^c_j} }\sinh(\theta_{2\lambda^c_j}) \right], \, \,  \, \theta_{2\lambda^c_j}= \left(\frac{\mu +2\lambda^c_j}{D^c_j}\right)^{\frac12}.
 \end{eqnarray*}
For $\tilde w^0_2(x)$, since $0 \notin \sigma(\mathcal A_0)$ and using $\xi_2^0(x_M) =1$,  we have 
\begin{eqnarray*}
\tilde  w^0_2(x_M) = \alpha_p \alpha_m 
 f^{\prime\prime}( p^\ast(x_M, D^c_j))  \mathcal G_2(x_M)
 \Big [1-  \alpha_p \alpha_m f^\prime( p^\ast(x_M, D^c_j)) \mathcal G_2(x_M)\Big]^{-1} ,
\end{eqnarray*}
where
\begin{eqnarray*}
&& \mathcal G_2(x_M) =
 \frac { \cosh^2(\theta x_M)}  {4 \mu \,  D_j^c \sinh^2(\theta)} 
  \left[1 + \frac 1{ \theta }\sinh(\theta) \right] \;  \quad \text{ with } \, \theta=(\mu/D_j^c)^{1/2}.
\end{eqnarray*}
Using the fact that $\xi^0_2(x_M) =1$ we can compute    
\begin{eqnarray*}
 \xi^0_1(x) &=&\frac {\alpha_m  f^\prime( p^\ast(x_M, D_j^c))}{ ((\mu+\lambda_j^c) D)^{1/2} \sinh(\theta_{\lambda_j^c})}
\Big[ \cosh(\theta_{\lambda_j^c} x) \cosh(\theta_{\lambda_j^c}(1-x_M))_{0<x<x_M}  \\
&&+   \cosh (\theta_{\lambda_j^c}(1-x))\cosh( \theta_{\lambda_j^c} x_M)_{x_M<x<1} \Big].
 \end{eqnarray*}
To define the solution of \eqref{EW_problem_omega_adjont} with $\ve= 0$ we note that   $-\lambda_j^c\notin\sigma(\mathcal A_0)$ and  obtain 
\begin{eqnarray*}
\xi_2^{\ast, 0}(x)  =  \alpha_m G_{-\lambda_j^c+\mu}(x, x_M)  f^\prime( p^\ast_0) \xi_1^\ast(x_M) =   
 \frac {\alpha_m  f^\prime( p^\ast_0(x_M, D^c_j)) \xi_1^{\ast,0}(x_M)}{ ((\mu-\lambda^c_j) D_j^c)^{1/2} \sinh(\theta_{-\lambda^c_j})}\times \\
  \left[\cosh(\theta_{-\lambda^c_j} x) \cosh(\theta_{-\lambda_j^c}(1-x_M))_{x<x_M}  
+   \cosh (\theta_{-\lambda^c_j}(1-x))\cosh( \theta_{-\lambda^c_j} x_M)_{x_M<x} \right],
\end{eqnarray*}
With $l = 1/2$  we have that $\xi_1^{\ast,0}$ has the form  
\begin{eqnarray*}
 \xi_1^{\ast,0}(x) =
    \frac{\alpha_p  \alpha_m} 2  
     \frac{ f^\prime( p^\ast_0(x_M, D^c_j))\cosh(\theta_{-\lambda_j^c} x_M) }{(\mu-\lambda^c_j) D \sinh^2(\theta_{-\lambda^c_j})}   \xi_1^{\ast,0}(x_M)  
 \times \hspace{2 cm } \\
 \times \left[
  \cosh(\theta_{-\lambda^c_j}(1-x))\left( \cosh(\theta_{-\lambda^c_j}) \left[x- \frac 12\right]_{+} 
  - \frac {\sinh(\theta_{-\lambda^c_j}(1- 2 x))_{x\geq \frac12}} {2\theta_{-\lambda^c_j}}\right)  \right.
   \\
\left. +   \cosh(\theta_{-\lambda^c_j} x)
\left(\left[1- \max\left\{x, \frac 12\right\}\right] +\frac{\sinh\left(2\theta_{-\lambda^c_j}\left[1- \max\left\{x, \frac 12\right\}\right]\right)} {2\theta_{-\lambda^c_j}}\right)
 \right]
\end{eqnarray*}
 where  $\theta_{-\lambda^c_j} = \left(\frac{\mu - \lambda^c_j}{D_c} \right)^{1/2}$.  We define $\xi_1^{\ast,0} (x_M)  $ in such a way that 
$$\langle \xi^0, \xi^{\ast,0}\rangle= \int_0^1\left(\xi_1^0(x)\overline{\xi_1^{\ast,0}}(x) +\xi_2^0(x)\overline{\xi_2^{\ast,0}}(x) \right) dx  =1$$
and  considering  that $x_M < 1/2$, we obtain
\begin{eqnarray*}
&&  \overline {\xi_1^{\ast,0}}(x_M) = \left[\frac {\alpha_m^2  \alpha_p [ f^\prime( p^\ast_0)]^2 \cosh(\theta_{\lambda^c_j} x_M) }{ ((\mu+\lambda^c_j) D)^{3/2} \sinh^3(\theta_{\lambda^c_j})} \right]^{-1}   \left[ \left(\frac12 +\frac{\sinh(\theta_{\lambda^c_j})}{2\theta_{\lambda^c_j}} \right) \left[  \cosh( \theta_{\lambda^c_j} x_M) \times \right.  \right. 
\\ && \left.  \int_{x_M}^{1/2} 
  \cosh (\theta_{\lambda^c_j}(1-x))  \cosh(\theta_{\lambda^c_j} x)
 dx + \cosh(\theta_{\lambda^c_j}(1-x_M))  \int_0^{x_M}  
   \cosh^2(\theta_{\lambda^c_j} x) dx \right]\\
   &&
+   \cosh(\theta_{\lambda^c_j} x_M)
\int_{1/2}^1  \cosh (\theta_{\lambda^c_j}(1-x)) 
 \left[  \cosh(\theta_{\lambda^c_j} x)
\left(1- x +\frac {\sinh(2\theta_{\lambda^c_j}(1- x)) }{2\theta_{\lambda^c_j}} \right) \right. \\
&& \left. \left.  + 
  \cosh(\theta_{\lambda^c_j}(1-x))\left( \cosh(\theta_{\lambda^c_j}) \left(x-\frac 12\right)  
 - \frac {\sinh(\theta_{\lambda^c_j}(1- 2 x))}{2\theta_{\lambda^c_j}}  \right) 
    \right] dx \right]^{-1}\; ,
\end{eqnarray*}
where $\theta_{\lambda^c_j} = ((\mu+ \lambda^c_j)/D)^{1/2}$ and   $j=1,2$.

Carrying out all calculations in Matlab, for the critical value of the bifurcation parameter $D^c_1\approx 3.117 \times 10^{-4}$ we obtain 
$b_{1,0}\approx   -0.041 - 0.015i$. Thus since $\mathcal Re(b_{1,0}) <0$ we have by continuity and strong convergence that  the Hopf bifurcation at $D^c_{1,\ve}$ is supercritical and we have a stable family of periodic solutions bifurcating from the steady state into the region $D>D^c_{1, \ve}$ where the stationary solution is unstable, i.e. $\nu =1$.
For the second critical value  $D^c_2\approx 7.885 \times 10^{-3}$, the calculated value is  
$b_{2,0} \approx - 0.0659 - 0.0175 i$ and, since  $\mathcal Re(b_{2,0}) <0$, the Hopf bifurcation  at $D^c_{2,\ve}$ is also supercritical and stable periodic orbits bifurcate into the region $D<D^c_{2, \ve}$ where the stationary solution is unstable, i.e. $\nu =-1$. 
\end{proof}
 
 
The amplitude equation can also be derived using central manifold theory and the corresponding  normal form for the  system of partial differential equations, see Haragus \& Iooss \cite{Haragus}. To apply the known results we shall shift the values of critical parameters and stationary solutions to zero, i.e.  $\tilde D=D- D^c_{j, \ve}$ and   $\tilde m(t,x) = m(t,x)- m^\ast_\ve(x,D)$, $\tilde p(t,x) = p(t,x) - p^\ast_\ve(x,D)$, where 
$m^\ast_\ve(x,D)$ and $p^\ast_\ve(x,D)$ are the stationary solutions of  \eqref{dynamic_problem}. Then   \eqref{dynamic_problem} can be written as: 
\begin{eqnarray}\label{main_problem_zero}
\partial_t u = \mathcal { A}_{D_{j,\ve}^c} u + \tilde F(u,\tilde D), 
\end{eqnarray}
where  $u(t,x)= (\tilde m(t,x), \tilde p(t,x))$ with  
\begin{eqnarray*}
\mathcal { A}_{D_{j,\ve}^c} =
\begin{pmatrix}
 D^c_{j, \ve}\dfrac{\partial^2}{\partial x^2} - \mu  &\qquad   \alpha_m f^\prime(p^\ast_\ve(x,D^c_{j,\ve}))\delta^\ve_{x_M}(x) \\
 \alpha_p \, g(x)  & \qquad   D^c_{j,\ve}\dfrac{\partial^2}{\partial x^2} - \mu
 \end{pmatrix}
\end{eqnarray*}
and
\begin{eqnarray*}
\tilde F(u,\tilde D) = \begin{pmatrix}
 \alpha_m \left[f(\tilde p + p^\ast_\ve(\tilde D))  -  f(p^\ast_\ve(\tilde D)) -  f^\prime(p^\ast_\ve(D^c_{j,\ve}))\, \tilde p \right]\delta^\ve_{x_M}+ \tilde D \partial_x^2 \tilde m\\
 \\
\tilde D \partial_x^2 \tilde p
\end{pmatrix} \;,
\end{eqnarray*}
where $p^\ast_\ve(\tilde D)=p^\ast_\ve(D^c_{j,\ve}+\tilde D)$. 
 By Theorem 3.3 in Haragus \& Iooss \cite{Haragus}, using the results of Theorem \ref{Hopf_existence} and the regularity of $f$ and of the stationary solution $u^\ast_\ve(x,D)=(m^\ast_\ve(x,D),p^\ast_\ve(x,D))$, we conclude that the system \eqref{main_problem_zero} possesses a two-dimensional centre manifold for sufficiently small $\tilde D$.   
 The equations in \eqref{main_problem_zero} reduced to the central manifold  can be transformed by the polynomial change of variables  in  the normal form \cite{Haragus,Hassard}
 \begin{equation}\label{normal_form_1}
 \frac{dA}{dt} =\lambda_{j,\ve}^c \,  A  +  a_{j,\ve}\, \tilde D\, A + b_{j, \ve} \, A \, |A|^2 + O(|A|(|\tilde D| + |A|^2)^2) ,
 \end{equation}
for $j=1,2$. The solutions of \eqref{main_problem_zero} on the centre manifold  are then of the form 
  \begin{equation}\label{central_manifold_u_normal}
 u = A \xi + \overline{A\xi} + \Phi( A, \bar A, \tilde D), \qquad A \in \mathbb C, 
 \end{equation}   
 where $\xi=(\xi_1, \xi_2)$ is an eigenvector for  the eigenvalue $\lambda_{j,\ve}^c$ and    for $\Phi$ a polynomial ansatz can be made:
 $$
 \Phi (A, \bar A, \tilde D) = \sum_{r, s, q} \Phi_{rsq} A^r \bar{A}^s \tilde D^q \; , 
 $$
 with $\Phi_{100} =0$, $\Phi_{010} =0$ and $\Phi_{rsq} = \overline{\Phi}_{srq}$.
 Substituting the form \eqref{central_manifold_u_normal} for $u$ into equations \eqref{main_problem_zero}, we obtain 
 $$
 (\xi + \partial_A \Phi ) \frac{ dA}{ dt} +
   (\overline\xi + \partial_{\bar A} \Phi ) \frac{ d\bar{A}}{ dt} =  \mathcal A_{D_{j,\ve}^c} (A \xi + \overline{A\xi} + \Phi) +  \tilde F(A\xi +  \overline{A\xi} + \Phi , \tilde D)\; .
 $$
Considering orders of $\tilde D A$, $A^2$, $A\bar A$, $A^2 \bar A$, implies the equations:
\begin{eqnarray*} 
 - \mathcal{ A}_{D_{j, \ve}^c}  \Phi_{001} & = &   \partial_{\tilde D} \tilde F(0, 0) , \\
 a_{j,\ve} \xi  + ( \lambda^c_{j,\ve} - \mathcal{A}_{D_{j, \ve}^c} ) \Phi_{101}& =&  \partial_{u} \partial_{\tilde D} \tilde F(0, 0) \xi+ 
  \partial_u^2 \tilde F(0, 0) (\xi, \Phi_{001}), \\
(2\lambda^c_{j, \ve} - \mathcal{ A}_{D_{j, \ve}^c} ) \Phi_{200} &=& \frac 12 \partial_u^2 \tilde F(0, 0) (\xi, \xi), \\
 - \mathcal{ A}_{D_{j, \ve}^c}  \Phi_{110} & = &   \partial_u^2 \tilde F(0, 0) (\xi, \bar \xi), \\
 b_{j,\ve} \xi + (\lambda^c_{j,\ve}  - \mathcal{A}_{D_{j, \ve}^c} ) \Phi_{210} & = & \partial_u^2 \tilde F(0, 0) (\bar \xi, \Phi_{200})  + 
  \partial_u^2\tilde  F(0, 0) (\xi, \Phi_{110}) \\
  && + \frac 12 \partial_u^3\tilde  F(0, 0) (\xi, \xi, \bar \xi).
\end{eqnarray*}
 We have 
$ \partial_{\tilde D} \tilde F(u^\ast, 0) = (0, 0)^T$ together with 
$$
\partial_u\partial_{\tilde D} \tilde F (0, 0)\,  \xi= 
 \begin{pmatrix}
\dfrac{d^2\xi_1}{dx^2}+  \alpha_m f^{\prime\prime}(p^\ast_\ve(x,D^c_{j, \ve})) \partial_{\tilde D} p^\ast(x,D^c_{j, \ve}) \delta^\ve_{x_M}(x) \xi_2\\
\\
 \dfrac{d^2 \xi_2}{dx^2}
\end{pmatrix}
$$
and multilinear forms $\partial_u^2 \tilde F(0, 0) $ and $\partial_u^3 \tilde F(0, 0)$ are defined as
$$
\partial^2_u \tilde F (0, 0)(\xi, \xi) = 
 \begin{pmatrix}
 \alpha_m f^{\prime\prime}(p^\ast_\ve(x,D^c_{j,\ve})) \delta^\ve_{x_M}(x)\xi_2^2 \\
 \\
 0 
\end{pmatrix}
$$
and 
$$
\partial^3_u\tilde  F (0, 0)(\xi, \xi, \bar \xi) = 
 \begin{pmatrix}
 \alpha_m f^{\prime\prime\prime}(p^\ast_\ve(x,D^c_{j, \ve})) \delta^\ve_{x_M}(x)\xi_2^2 \bar \xi_2\\
 \\
 0 
\end{pmatrix} .
$$
Since $ \partial_{\tilde D} \tilde F(0, 0) = (0, 0)^T$ and $0 \notin \sigma(\mathcal{A}_{D_{j, \ve}^c} )$, we obtain that 
$\Phi_{001} =0$. Applying the Fredholm alternative for the solvability of equations for $ \Phi_{101} $ and  $\Phi_{210}$, we obtain the same expressions for coefficients $a_{j,\ve}$ and $b_{j,\ve}$ as from the weakly nonlinear analysis. 

The relation between the normal form and the equation for the amplitude obtained from the nonlinear analysis can be understand by introducing  in the normal form  \eqref{normal_form_1}  the assumption, using  in the nonlinear analysis,   that  the amplitude depends  on the fast time scale  $t$ and the slow time scale $T$, i.e. $A= A(t,T)$. Taking into  account  $\tilde m(t,x) =m(t,x)-m^\ast_\ve(t,x)\approx\delta $, $\tilde p(t,x) = p(t,x)-p^\ast_\ve(t,x)\approx\delta$, $\tilde D\approx \delta^2\nu$ and $T= t/\delta^2$, implies
\begin{eqnarray*}
\delta \frac {d A}{dt} + \delta^3 \frac{d A}{dT} = 
\lambda^c_{j, \ve}  \delta A + a_{j,\ve} \delta^3 \nu  A+ \delta^3 b_{j,\ve} A|A|^2 + \delta^5 O(|A|(|\nu| + |A|^2)^2).
\end{eqnarray*}
Then for the terms of  orders $\delta$ and $\delta^3$, we obtain the equations derived using weakly-nonlinear analysis, i.e.
\begin{eqnarray*}
\frac {dA}{dt} = \lambda^c_{j,\ve} A \quad \text{ and } \quad 
 \frac{dA}{dT} = a_{j,\ve}\,  \nu \,  A +  b_{j,\ve} \, A|A|^2 \; .
\end{eqnarray*}

\section{Discussion and Conclusions}

Transcription factors play a vital role in controlling the levels of proteins and mRNAs within cells, and are involved in many key processes such as cell-cycle regulation and apoptosis. Such systems are often referred to as gene regulatory networks (GRNs). Those transcription factors which down-regulate (repress/suppress) the rate of gene transcription do so via negative feedback loops, and such intracellular negative feedback systems are known to exhibit oscillations in protein and mRNA levels. 

In this paper we have analysed a mathematical model of the most basic gene regulatory network consisting of a single negative feedback loop between a protein and its mRNA - the Hes1 system. Our model consisted of a system of two coupled nonlinear partial differential equations describing the spatio-temporal dynamics of the concentration of hes1 mRNA, $m(x,t)$, and Hes1 protein, $p(x,t)$, describing the processes of transcription (mRNA production) and translation (protein production). Numerical simulations demonstrated the existence of oscillatory solutions as observed experimentally \cite{Hirata}, with the indication that the periodic orbits arose from supercritical Hopf bifurcations at two critical values of the bifurcation parameter $D_1^c$ and $D_2^c$. 
 These results were then proved rigorously, demonstrating that the diffusion coefficient of the protein/mRNA acts as a bifurcation parameter and showing that the spatial movement of the molecules alone is sufficient to cause the oscillations.

Our result is in line with recent experimental findings \cite{Hanisch,Hoyle} where the longest delay in several transcription factor systems was due to mRNA export from the nucleus rather than delays associated with the process of gene splicing. These results are also in line with other data which suggest that transcripts have a restricted rate of diffusion according to their mRNP (messenger ribonucleoprotein) composition \cite{Grunwald,Mor,Siebrasse}. It is not unreasonable to assume that further delays in the export process could also occur due to docking of transcripts with the pores of the nuclear membrane, and transcript translocation across the nuclear pores into the cytoplasm. These experimental observations and the main result of this paper (molecular diffusion causes oscillations) confirm the importance of modelling transcription factor systems where negative feedback loops are involved, using explicitly spatial models.

\section*{Acknowledgments} 

MAJC and MS gratefully acknowledge the support of the ERC Advanced Investigator Grant 227619, ``M5CGS - From Mutations to Metastases: Multiscale Mathematical Modelling of Cancer Growth and Spread". MS would also like to thank the support from the Mathematical Biosciences Institute at the Ohio State University and NSF grant DMS0931642.

\bibliographystyle{acm}
\bibliography{Stability} 

\end{document}